\documentclass[letterpaper]{article}

\usepackage[english]{babel}
\usepackage[utf8x]{inputenc}
\usepackage[T1]{fontenc}
\usepackage{fullpage}

\usepackage{amsmath}
\usepackage{graphicx}
\usepackage[colorinlistoftodos]{todonotes}
\usepackage[colorlinks=true, allcolors=blue]{hyperref}
\usepackage{tikz}
\usepackage{amsthm}
\newtheorem{lemma}{Lemma}

\newtheorem{theorem}{Theorem}
\newtheorem{pro}{Proposition}
\newtheorem{cor}{Corollary}

\newtheorem{question}{Question}

\usepackage{amssymb}

\usepackage{subcaption}
\usepackage{multicol}
\usepackage[justification = centering]{caption}
\usepackage{float}

\title{Weak Dynamic Coloring of Planar Graphs}
\date{\vspace{-5ex}}
\begin{document}
\maketitle

\author{ Caroline Accurso$^{1,5}$, \and
	Vitaliy Chernyshov$^{2,5}$, \and 
	Leaha Hand$^{3,5}$, \and 
	Sogol Jahanbekam$^{2,4,5}$, and \and
	Paul Wenger$^{2}$}

\footnotetext[1]{Department of Mathematics, DeSales University, Center Valley, PA ; {\tt ca3070@desales.edu}.}
\footnotetext[2]{School of Mathematical Sciences, Rochester Institute of Technology, Rochester, NY; {\tt vac4329@mail.rit.edu, sxjsma@rit.edu, pswsma@rit.edu.}}
\footnotetext[3]{Department of Mathematics, Boise State University, Boise ID; {\tt leahahand@u.boisestate.edu.}}
\footnotetext[4]{Research supported in part by NSF grant CMMI-1727743.}  
\footnotetext[5]{Research supported in part by NSF grant REU-1659075.}

\begin{abstract}
The {\it $k$-weak-dynamic number} of a graph $G$ is the smallest number of colors we need to color the vertices of $G$ in such a way that each vertex $v$ of degree $d(v)$ sees at least \rm{min}$\{k,d(v)\}$ colors on its neighborhood. We use reducible configurations and  list coloring of graphs to prove that all planar graphs have 3-weak-dynamic number at most 6. %We also use the Probabilistic methods to determine upperbounds for r-weak dynamic number of regular graphs. 
\end{abstract}
%The \it{r-dynamic chromatic number} of a graph $G$ is the smallest number of colors we can use to color the vertices of $G$ properly in such a way that each vertex $v$ of degree $d(v)$ sees at least \rm{min}$\{r,d(v)\}$ colors on its neighborhood. 

%The chromatic number of a hypergraph $H$ is the smallest number of colors we can use to color the vertices of $H$ in such a way that each edge of size at least 2 sees at least 2 different colors. 


\section{Introduction}

% Let $G$ be a graph and let $v$ be a vertex of $G$. We define $d_G(v)$ to be the degree of the vertex $v$ in $G$. Throughout the paper when there is no fear of confusion we use $d(v)$ in place of $d_G(v)$. We denote the minimum degree of $G$ and the maximum degree of $G$ by $\delta(G)$ and $\Delta(G)$, respectively. 

 % Let $V(G)$ and $E(G)$ denote the vertex set of $G$ and the edge set of $G$, respectively.

 A \textit{proper coloring} of $G$ is a vertex coloring of $G$ in which adjacent vertices receive different colors. The \textit{chromatic number} of $G$, written as $\chi(G)$, is the smallest number of colors needed to find a proper coloring of $G$.   For notation and definitions not defined here we refer the reader to \cite{w}. 
 
 %A \textit{planar graph } is a graph that can be embedded on the plane with no edge crossing.

 A $k$-\textit{dynamic coloring} of a graph $G$ is a proper coloring of $G$   in such a way that each vertex sees at least $\rm{min}\{d(v),k\}$ colors in its neighborhood. The  $k$-\textit{dynamic chromatic number} of a graph $G$, written as $\chi_k(G)$, is the smallest number of colors needed to find an $k$-dynamic coloring of $G$. Dynamic coloring of graphs was first introduced by Montgomery in \cite{m}. 
 
 Montgomery \cite{m} conjectured that $\chi_2(G)\leq \chi(G)+2$, for all regular graphs $G$. Montgomery's conjecture was shown to be true for some families of graphs including bipartite regular graphs  \cite{agj1}, claw-free regular  graphs \cite{m}, and regular graphs with diameter at most 2 and chromatic number at least 4 \cite{alishahi}.   For all integers $k$, Alishahi \cite{alishahi} provided a regular graph $G$ with $\chi_2(G)\geq \chi(G)+1$ and $\chi(G)=k$. In \cite{alishahi2}, Alishahi proved that $\chi_2(G)\leq 2\chi(G)$ for all regular graphs $G$. Later Bowler et al. \cite{counterexample} disproved the Montgomery's conjecture by showing that Alishahi's bound is best possible. For all  integers $n$ with $n\geq 2$, they found  a regular graph $G$ with $\chi(G)=n$ but $\chi_2(G)=2\chi(G)$.  Other upper bounds have also been determined for the $k$-dynamic chromatic number of regular graphs and general graphs. See for example \cite{alishahi2,da,jkow,t}.

  %Another constraint of dynamic coloring is that the coloring of the graph must be proper, that is no two adjacent vertices are allowed to be the same color. While this type of coloring has been studied rigorously,
  
In   this paper we look at a weaker form of dynamic coloring in which we do not look at the constraint that the coloring must be proper. We refer to this type of coloring as a {\it weak-dynamic coloring}. Therefore a $k$-\textit{weak-dynamic coloring} of a graph $G$ is a coloring of the vertices of $G$ in such a way that each vertex $v$  sees at least $\rm{min}\{d(v),k\}$ colors in its neighborhood. We define $k$\textit{-weak-dynamic
 number} of $G$, written as $wd_k(G)$,  to be the smallest number of colors needed to obtain a $k$-weak-dynamic coloring of $G$. 
 
By an  observation in \cite{jkow} we have $\chi_k(G)\leq \chi(G)wd_k(G)$, because we can associate to each vertex of $G$ an ordered pair of colors in which the first color comes from a proper coloring of $G$ and the second color comes from a $k$-weak-dynamic coloring of $G$, to obtain a $k$-dynamic coloring of $G$. 

A \textit{proper coloring} of a hypergraph is a coloring of its vertices in such a way that each hyperedge sees at least two different colors. For a graph $G$, let  $H$ be the hypergraph with vertex set $V(G)$  whose edges are the vertex neighborhoods in $G$. When $\delta(G)\geq 2$,  any 2-weak-dynamic coloring of $G$ corresponds to a proper coloring of $H$ and vice versa. %Therefore if $G$ is $k$-regular, then its corresponding hypergraph $H$ is $k$-uniform. The proper coloring of these classes of hypergraphs 
  
In this paper we study weak-dynamic coloring of planar graphs. Kim et al. \cite{kim} proved that $\chi_2(G)\leq 4$ for all planar graphs $G$ with no $C_5$-component. Note also that we can find a 2-weak-dynamic coloring of $C_5$ using only 3 colors. Therefore the inequality $wd_2(G)\leq \chi_2(G)$ implies that all planar graphs have 2-weak-dynamic coloring at most 4. We also know that the upper bound $4$ for the 2-weak-dynamic coloring of planar graphs is best possible, as $wd_2(G)=4$ when $G$ is a subdivision of $K_4$. Our aim in this paper is to obtain an upper bound for  $wd_3(G)$ when $G$ is a planar graph. We prove the following theorem. 

\begin{theorem}\label{main}
Any planar graph $G$ satisfies $wd_3(G)\leq 6$.
\end{theorem}

In order to prove Theorem \ref{main}, we first study an edge-minimal counterexample $G$ to the statement of the theorem. In Section \ref{preliminary} we provide some tools we need during our proofs.  In Section \ref{reducible} we determine some configurations that do not exist in $G$; we call these {\it reducible configurations}.  In Section \ref{proof} we use the reducible configurations we obtain in Section \ref{reducible}  and the the tools we introduce  in Section~\ref{preliminary} to obtain a 3-weak-dynamic coloring of $G$ using 6 colors, which gives us a contradiction showing that no counterexample exists.

%This paper will argue how the $k$-weak-dynamic coloring of any planar graph $G$ for $k=3$ has chromatic number (referred to as $wd_3(G)$ henceforth) $wd_3(G)\leq 6$. In order to prove our claim we utilized the discharging method, which allows us to assign a charge to each vertex equal to the degree of the vertex. Likewise, a charge is given to each face equal to its length, that is how many vertices are found in the face. In order to use the discharging method, we had to first obtain our smallest counter-example graph, the graph with the minimum number of configurations - faces and vertices of certain degrees - and show that said graph cannot be planar. In order to do this, we looked at dozens of configurations that we were able to reduce. In order to reduce a configuration, or claim that it cannot exist in our smallest counter-example, we had to prove that it was possible to use at most six colors to fill in the configuration while still satisfying the conditions for a weak-dynamic coloring. In the next section, we show which configurations are reducible and how we came to that conclusion.

\section{Preliminary Tools}\label{preliminary}

A $d$-vertex in $G$ is a vertex of degree $d$ in $G$. A $d^+$-vertex in $G$ is a vertex of degree at least $d$ in $G$ and a $d^-$-vertex in $G$ is a vertex of degree at most $d$ in $G$.  A $d$-neighbor of a vertex $v$ in $G$ is a neighbor of $v$ having degree $d$.  Similarly, $d^+$-neighbors of $v$ have degree at least $d$, and $d^-$-neighbors of $v$ have degree at most $d$. For a vertex $v$, $N_G(v)$ (or simply $N(v)$) is the set of neighbors of $v$ in $G$. We define $N^2(v)$ to be the set of vertices in $G$ having a common neighbor with $v$. Let $c$ be a vertex coloring of $G$ and $A\subseteq V(G)$. We define $c(A)$ to be the set of colors on vertices in $A$.

During the proof of Theorem \ref{main}, we correspond an edge-minimal counterexample graph $G$ to  an auxiliary graph $H$ having the same vertex set as $G$ but with different set of edges. We build $H$ in such a way that any proper coloring of $H$ corresponds to a 3-weak-dynamic coloring of $G$. Hence for the rest of the proof, our aim would be to find a proper coloring of $H$ using 6 colors. To fulfill the aim we use the following  results on proper coloring of graphs and on planar graphs.

\begin{theorem}[Four-Color Theorem, Appel and Haken \cite{4-color-theorem}]\label{4-color-theorem}
Any planar graph has chromatic number at most 4. 
\end{theorem}

\begin{theorem}[Wagner's Theorem, Wagner \cite{wa}]\label{minor-planar}
A graph $G$ is planar if and only if $K_{3,3}$ and $K_5$ are not minors of $G$. 
\end{theorem}

For each vertex $v$ in a graph $G$, let $L(v)$ denote a list of colors available at $v$. A \textit{list coloring} of $G$ is a proper coloring $f$ such that $f(v)\in L(v)$ for each vertex $v$ of $G$. We say that $G$ is \textit{$L$-choosable} if it has a list coloring under $L$. We say that $G$ is \textit{degree-choosable} if $G$ has a list coloring for all lists $L$ with $|L(v)|=d(v)$.  A graph $G$ is {\it $2$-connected} if it is connected and the removal of any vertex from $G$ leaves it connected. A \textit{block} of $G$ is a maximal 2-connected subgraph of $G$ or a cut-edge. Not all graphs are degree-choosable. For example, odd cycles and complete graphs are not degree choosable. The following result classifies all graphs $G$ that are degree-choosable.

\begin{theorem}[Borodin \cite{borodin} and Erd\H os, Rubin, and Taylor \cite{erdos}]\label{list}
Let $G$ be a connected graph having a block that is not an odd cycle nor a complete graph. The graph $G$ is degree-choosable.
\end{theorem}

Theorem \ref{list} implies the following Corollary. 

\begin{cor}\label{cor-list}
Let $G$ be a connected graph and $L$ be a list assignment on the vertices $x\in G$ such that $|L(x)|\ge d(x)$ for all $x$.
If there each vertex $v\in V(G)$ such that $|L(v)|>d(v)$, then $G$ is $L$-choosable.
\end{cor}

%\begin{cor}\label{cor-list}
%Let $G$ be a connected graph and $L$ be a list assignment on the vertices of $G$ such that each vertex has a list of size at least its degree and at least a vertex $v$ in $G$ has a list of size larger than is degree, then $G$ is $L$-choosable.
%\end{cor}

\begin{proof}
Add a vertex $u$, an edge $uv$ to $G$, and add a pendant even cycle $C$ to $u$ in this graph. Give all vertices of $C$ a list of size $3$ and keep the list $L$ on other vertices of $G$. Let $H$ be the resulting graph and $L'$ be the list we defined on vertices of $H$. Since $C$ is a block of $H$, by Theorem~\ref{list} the graph $H$ is $L'$-choosable, which implies that $G$ is $L$-choosable. 
\end{proof}

The following propositions are known results on proper list coloring of complete graphs and odd cycles.

\begin{pro}\label{complete-graph}
Let $L$ be a list assignment on the vertices of the complete graph $K_n$ with vertex set $\{v_1,\ldots,v_n\}$ in such a way that $|L(v_i)|=n-1$ for each $i$ and $L(v_1)\neq L(v_k)$. The graph $K_n$ is $L$-choosable.
\end{pro}

\begin{proof}
First color $v_1$ by a color in $ L(v_1)-L(v_n)$. Now choose appropriate colors for vertices $v_2,\ldots,v_{n-1}$ from their lists respectively in such a way that adjacent vertices get different colors. At each step  the vertex $v_i$ must have a color different from the color of at most $n-2$ other vertices. Having $|L(v_i)|=n-1$, we are able to choose these colorings. Finally since the color of $v_1$ does not belong to $L(v_n)$, it is enough to choose a color for $v_n$ to be a color in $L(v_n)$ and different from the colors of $v_2,\ldots,v_{n-1}$ to obtain a proper coloring of $K_n$. 
\end{proof}

\begin{pro}\label{odd-cycle}
Let $L$ be a list assignment on the vertices of an odd cycle $C$ with vertices $v_1,\ldots, v_k$ so that $|L(v_i)|=2$ for each $i\in[k]$ and $L(v_1)\neq L(v_k)$. The cycle $C$ is $L$-choosable.
\end{pro}

\begin{proof}
First color $v_1$ by a color in $ L(v_1)-L(v_k)$. Now choose appropriate colors for vertices $v_2,\ldots,v_{k-1}$ from their lists respectively in such a way that adjacent vertices get different colors. At each step the vertex $v_i$ must have a color different from the color of $v_{i-1}$. Having $|L(v_i)|=2$, we are able to choose these colorings. Finally choose a color for $v_k$ to be a color in $L(v_k)$ and different from the color of $v_{k-1}$ to obtain a proper coloring of $C$. 
\end{proof}

The following Proposition is an excercie in \cite{w}.  

\begin{pro}\label{pro-walk}
Let $W$ be a closed walk of a graph $G$ in such a way that no edge is repeated immediately in $W$. The graph $G$ contains a cycle.
\end{pro}

\begin{proof}
We prove the assertion by applying induction on the length of $W$. Note that such a closed walk $W$ cannot have length 1 or 2.
 If $W$ has length 3, then it is a triangle, which is a cycle, as desired. Now suppose $W$ is a walk of length at least 4 in which no edge is repeated immediately. If there is no vertex repetition other than the first vertex, then $W$ is a cycle, as desired. Hence
suppose there is some other vertex repetition. Let $W'$ be the portion of $W$ between
the instances of such a repetition. In case we have several options for $W'$, we choose one to be the shortest such portion. The walk $W'$ is a shorter closed walk
than $W$ and has the property that no edge is repeated immediately, since $W$ has this property. By the induction hypothesis,  the subgraph of $G$ over the edges of $W'$ has a cycle, and thus $G$ contains a cycle.
\end{proof}
%%%%%%%%%%%%%%%%%%%%%%%%%%%%%%%%%%
\section{Reducible Configurations}\label{reducible}
To prove Theorem \ref{main} we show that no counterexample exists to the statement of the theorem. Therefore we start by studying an edge-minimal counterexamples  $G$ of the theorem. If there are several such counterexamples, we choose $G$ to be a graph with the smallest number of vertices. 

During the proofs of the lemmas in this section, we look at a particular configuration that exists in $G$.  We use deletion of edges and vertices, and sometimes contracting  edges to obtain a new graph $H$ with smaller number of edges than $G$. As a result, the graph $H$ is not a counterexample any more. Hence $wd_3(H)\leq 6$. To obtain a contradiction, we use a $3$-weak-dynamic coloring of $H$ to find a 3-weak-dynamic coloring of $G$ using 6 colors. % Sometimes we simply extend the coloring of $H$ to a weak 3-dynamic coloring of $G$ by assigning appropriate colors to those vertices in $G$ that were deleted to obtain $H$. Sometimes we uncolor several vertices of $H$ first and then do the extension. 

% if a vertex has not yet been colored, it does not count as a dependency on other vertices being colored. Once it is assigned a color, it restricts the other vertices in particular situations. 

In a partial coloring of the vertices of a graph $G$, once a vertex has satisfied the requirements for a 3-weak-dynamic coloring (it sees at least three different colors in its neighborhood) we say the vertex is \textit{satisfied}. %Once a vertex is satisfied, its colored neighbors no longer are dependencies on other vertices that have yet to be colored.

In the following we determine a set of reducible configurations via different lemmas.

\begin{lemma}\label{lemma:AdjacentDegree2s}
The edge-minimal graph $G$ with $wd_3(G)>6$ satisfies $\delta(G)\geq 2$. Moreover $G$ has no $2$-vertex with a $3^-$-neighbor.
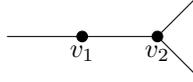
\begin{figure} [H]
\centering
\begin{tikzpicture}
\draw (-1,0) -- (1,0);
\draw (1.5,.5) -- (1,0) -- (1.5, -.5);
\filldraw (0,0) circle (2pt) node[anchor=north]{$v_1$};
\filldraw (1,0) circle (2pt) node[anchor=north] {$v_2$};
\end{tikzpicture}
\caption{A 2-vertex adjacent to a 3-vertex.}\label{Fig-2-3}
\end{figure}
\end{lemma}

\begin{proof}
By the choice of $G$ the graph $G$ is connected. Therefore it has no isolated vertex. If $G$ has a vertex $u$ of degree $1$, then  $wd_3(G-u)\leq 6$, as $G-u$ has fewer edges than $G$. Therefore there exists a 3-weak-dynamic coloring of $G-u$ with colors $\{1,\ldots,6\}$. Extend this coloring by giving $u$ a color in $\{1,\ldots,6\}$ that is different from two colors in the second neighborhood of $u$. This new coloring is a 3-weak-dynamic coloring of $G$,   a contradiction. Hence $\delta(G)\geq 2$.

Now we prove that $G$ has no 2-vertex $v_1$ having a $3^-$-neighbor $v_2$. We prove $d(v_2)=3$ gives us a contradiction. The proof of the case that $d(v_2)=2$ is similar. Hence we suppose $d(v_2)=3$.  Let $H=G-\{v_1v_2\}$. Since $H$ has fewer edges than $G$, by the choice of $G$ we have $wd_3(H) \le 6$. Therefore, there exists $c:V(H)\rightarrow \{1,\ldots,6\}$ that is a 3-weak-dynamic coloring of $H$. We recolor $v_1$ and $v_2$ in $c$ to obtain a 3-weak-dynamic coloring of $G$.

Let $u_1$ be the other neighbor of $v_1$ in $G$ and let $u_2$ and $u_3$ be the other neighbors of $v_2$ in $G$.  Choose a color in $\{1,\ldots,6\}$ for $v_1$ that satisfies $v_2$ and $u_1$. Satisfying $v_2$ and $u_1$ requires at most four restrictions. Therefore a desired color for $v_1$ exists.  Similarly, choose a color in $\{1,\ldots,6\}$ for $v_2$ to be different from $c(u_1)$ and to satisfy $u_2$ and $u_3$. We have at most five restrictions for the coloring of $v_2$. With six available colors, a desired coloring for $v_2$ exists. Hence this new coloring is a 3-weak-dynamic coloring of $G$ with six colors, which is a contradiction.
\end{proof}

\begin{lemma}\label{4-4}
The edge-minimal graph $G$ with $wd_3(G)<6$ has no pair of adjacent vertices of degree at least 4. 
\end{lemma}

\begin{proof}
Suppose $uv\in E(G)$ with $d(u),d(v)\geq4$. By the choice of $G$, we have $wd_3(G-uv)\leq 6$. But any 3-weak-dynamic coloring of $G-uv$ is also a 3-weak-dynamic coloring of $G$, so we obtain a contradiction.
\end{proof}

%%%

%%%

%%%

%%%

%%%

%%%

\begin{lemma}\label{lemma:4-3-3-4Configuration}
The edge-minimal graph $G$ with $wd_3(G) >6$ does not contain distinct vertices  $v_1,v_2,v_3,v_4,v_5,v_6$ such that $v_1v_2,v_2v_3,v_3v_4,v_2v_5,v_3v_6\in E(G)$,  $d(v_1) \ge 4$, $d(v_4) \ge 4$ and $d(v_2)=d(v_3)=d(v_5)=d(v_6)=3$ 
\begin{figure}[H]
\centering
\begin{tikzpicture}[xscale = 0.75, yscale = 0.75]
\draw (-1.75,0) -- (-1,0);
\draw (-1,0) -- (0,0);
\draw (0,0) -- (1.5,0);
\draw (1.5,0) -- (2.5,0);
\draw (2.5,0) -- (3.25,0);
\draw (-1.75,-.75) -- (-1,0) -- (-1.75,.75);
\draw (3.25,.75) -- (2.5,0) -- (3.25,-.75);
\draw (0,0) -- (0,-1);
\draw (.45,-1.75) -- (0,-1) -- (-.45,-1.75);
\draw (1.5,0) -- (1.5,-1);
\draw (1.05,-1.75) -- (1.5,-1) -- (1.95,-1.75);
\filldraw (-1,0) circle (3pt) node[anchor=south] {$v_1$};
\filldraw (0,0) circle (3pt) node[anchor=south] {$v_2$};
\filldraw (1.5,0) circle (3pt) node[anchor=south] {$v_3$};
\filldraw (0,-1) circle (3pt) node[anchor=east] {$v_5$};
\filldraw (1.5,-1) circle (3pt) node[anchor=west] {$v_6$};
\filldraw (2.5,0) circle (3pt) node[anchor=south] {$v_4$};
\filldraw (-2,0) circle (1pt);
\filldraw (-2,0.25) circle (1pt);
\filldraw (-2,-0.25) circle (1pt);
\filldraw (3.5,0) circle (1pt);
\filldraw (3.5,0.25) circle (1pt);
\filldraw (3.5,-0.25) circle (1pt);
\end{tikzpicture}
\caption{Adjacent 3-vertices with $3$-neighbors and $4^+$-neighbors.}\label{Fig-4-3-3-4}
\end{figure}
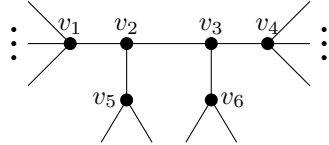
\end{lemma}

\begin{proof}
On the contrary suppose $G$ contains this configuration.  Let $H=G-\{v_2,v_3\}$.  Since $H$ has fewer edges than $G$, we have $wd_3(H) \le 6$.  Thus there exists
$c:V(H)\rightarrow \{1,\ldots,6\}$ that is a 3-weak-dynamic coloring of $H$.  We use $c$ to find a 3-weak-dynamic coloring of $G$.  To obtain this new coloring, we first recolor $c(v_5)$ and $c(v_6)$ and then choose appropriate colors for $v_2$ and $v_3$.

  Let $N(v_5)=\{v_2,v'_5,v''_5\}$ and $N(v_6)=\{v_3,v'_6,v''_6\}$. By Lemma \ref{lemma:AdjacentDegree2s}, the vertices $v'_5,v''_5,v'_6,v''_6$ have degree at least 3 in $G$. We first redefine $c(v_5)$ to be a color in $\{1,\ldots,6\}$ and different from $c(v_1)$, different from two distinct colors on $N(v'_5)$, and different from two distinct colors on $N(v''_5)$.  Since we require at most five restrictions for $v_5$, such a coloring for $v_5$ exists.  Next, we redefine $c(v_6)$ to to be a color in $\{1,\ldots,6\}$ and different from $c(v_4)$, different from two distinct colors on $N(v'_6)$, and different from two distinct colors on $N(v''_6)$.  Since we require at most five restrictions for $v_6$, such a coloring for $v_6$ exists. We have not colored $v_2$ and $v_3$ yet, but we know that vertices $v_1$ and $v_4$ are already satisfied, because they have degree at least 3 in $H$ and they are satisfied in $H$.

 We then choose $c(v_2)$ to be a color in $\{1,\ldots,6\}$ different from $c(v_4), c(v_6), c(v'_5),c(v''_5)$. Since we have  four restrictions for $c(v_2)$,  such a coloring for $v_2$ exists.  Last, we choose $c(v_3)$ to differ from $c(v_1),c(v_5),c(v'_6),c(v''_6)$.  Therefore we obtain a 3-weak-dynamic coloring of $G$ using six colors, which is a contradiction.
\end{proof}

\begin{lemma}\label{lemma:C_3bottomadjacentC_3}
The edge-minimal graph $G$ with $wd_3(G)>6$ does not contain a  3-face with vertices $v_1,v_2,v_3$  adjacent to a $3$-face with vertices $v_1,v_3,v_4$, where $d(v_1)=d(v_3)=3$.
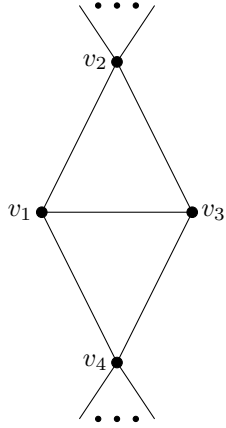
\begin{figure}[H]
\centering
\begin{tikzpicture}
\draw (-1,0) -- (0,2) -- (1,0) -- (0,-2) -- (-1,0) -- (1,0)
(-.5,2.75) -- (0,2) -- (.5,2.75) (-.5,-2.75) -- (0,-2) -- (.5,-2.75) ;
\filldraw (-.25,-2.75) circle (1pt);
\filldraw (0,-2.75) circle (1pt);
\filldraw (.25,-2.75) circle (1pt);
\filldraw (-.25,2.75) circle (1pt);
\filldraw (0,2.75) circle (1pt);
\filldraw (.25,2.75) circle (1pt);
\filldraw (-1,0) circle (2pt) node[anchor=east]{$v_1$};
\filldraw (0,2) circle (2pt) node[anchor=east]{$v_2$};
\filldraw (1,0) circle (2pt) node[anchor=west]{$v_3$};
\filldraw (0,-2) circle (2pt) node[anchor=east]{$v_4$};
%\node[anchor=south] at (-.5,2.75){$v_5$};
%\node[anchor=south] at (.5,2.75){$v_6$};
%\node[anchor=north] at (-.5,-2.75) {$v_7$};
%\node[anchor=north] at (.5,-2.75){$v_8$};
\end{tikzpicture}
\caption{Two adjacent triangles.}\label{Fig-Two adjacent triangles}
\end{figure}
\end{lemma}
\begin{proof}
On the contrary suppose $G$ contains this configuration. Contract the edge $v_1v_3$ into a single vertex $v_{1,3}$ %in $G$ 
and let $H$ be the resulting graph. Since $H$ has fewer edges than $G$, it follows that $wd_3(H)\leq 6$. Therefore there exists
$c: V(H) \to \{1,\ldots,6\}$ that is a 3-weak-dynamic coloring of $H$. To obtain a contradiction, we use $c$ to find a 3-weak-dynamic coloring of $G$. Note that the neighbors of the vertex $v_{1,3}$ in $H$ are $v_2$ and $v_4$, therefore we know $c(v_2)\neq c(v_4)$.

By Lemma \ref{lemma:AdjacentDegree2s}, we have $d_G(v_2)\geq 3$ and $d_G(v_4)\geq 3$. First suppose that $d_G(v_2)\geq 4$ and $d_G(v_4)\ge4$. In this case each of the vertices $v_2$ and $v_4$ has degree at least 3 in $H$. Hence $v_2$ sees at least three different colors on its neighborhood in $H$. As a result, $v_2$ sees at least two different colors on $N_H(v_2)-\{v_{1,3}\}$. Let's call these two colors $c_1$ and $c_2$. Similarly, suppose $c_3$ and $c_4$ are two different colors that appear on $N_H(v_4)-\{v_{1,3}\}$. We use the coloring of $c$ over $V(H)-\{v_{1,3}\}$ and then extend it to a 3-weak-dynamic coloring of $G$.

Choose $c(v_1)$ to be a color in $\{1,\ldots,6\}-\{c(v_2),c(v_4),c_1,c_2\}$. Then choose $c(v_3)$ to be a color in $\{1,\ldots,6\}-\{c(v_2),c(v_4),c_3,c_4\}$. The coloring $v_1$ is in such a way that the vertex $v_2$ gets satisfied and the coloring of $v_3$ is picked in such a way that $v_4$ becomes satisfied. Since the neighbors of $v_1$ get different colors and the neighbors of $v_3$ get different colors, this extension is indeed a 3-weak-dynamic coloring of $G$.

Now suppose that $d_G(v_2)=3$.  Let $c_1$ be the color of the neighbor of $v_2$ in $H$ that is different from $v_{1,3}$.   We use the coloring of $c$ over $V(H)-\{v_{1,3}\}$ and then extend it to a 3-weak-dynamic coloring of $G$.

Let $c_2$ and $c_3$ be colors on $N_H(v_4)-v_{1,3}$.  We choose $c_2$ to be different from $c_3$, when $d_G(v_4)\geq 4$. Otherwise $c_2=c_3$. Now  choose $c(v_3)$ to be a color in $\{1,\ldots,6\}-\{c(v_2),c(v_4),c_1,c_3,c_4\}$.  Then choose $c(v_1)$ to be a color in $\{1,\ldots,6\}-\{c(v_2),c(v_3),c(v_4),c_1,c_3\}$. These assignments satisfy the vertices $v_2$ and $v_4$. Since the neighbors of $v_1$ get different colors and the neighbors of $v_3$ get different colors, this extension is a 3-weak-dynamic coloring of $G$.
\end{proof}

\begin{lemma}\label{3-uniform triangle}
The edge-minimal graph $G$ with $wd_3(G)>6$ does not contain a triangle with vertices $v_1,v_2,v_3$, where $d(v_1)=d(v_2)=d(v_3)=3$. 
\end{lemma}

\begin{proof}
On the contrary suppose $G$ contains this configuration. For each $i$, let $N_G(v_i)-\{v_1,v_2,v_3\}=\{v'_i\}$.  By Lemma \ref{lemma:C_3bottomadjacentC_3} the vertices $v'_1,v'_2,v'_3$ are distinct.  Let $H=G-\{v_1,v_2,v_3\}$.  Since $H$ has fewer edges than $G$, we have $wd_3(H)\leq6$.  Thus there exists $c:V(H)\rightarrow\{1,\ldots,6\}$ that is a 3-weak-dynamic coloring of $H$.  We use $c$ to find a 3-weak-dynamic coloring of $G$.  By Lemma \ref{lemma:AdjacentDegree2s} we have $d_G(v'_1)\geq 3$, $d_G(v'_2)\geq 3$, and $d_G(v'_3)\geq 3$.  We consider two cases.

\begin{description}
\item [Case 1:]  $d_G(v_1')=d_G(v'_2)=d_G(v_3')= 3$.  
Let  $N_G(v'_i)-\{v_i\}=\{w_i,w'_i\}$.  We recolor $v_1',v'_2,v_3'$ and find appropriate colors for $v_1,v_2,v_3$.  We will call the set of vertices that we plan to color or recolor $S$.  Thus, $S=\{v_1,v_2,v_3,v_1',v'_2,v_3'\}$. 

Now we study the restrictions we must consider  for the coloring on $S$ to make sure that a 3-weak-dynamic coloring of $G$ is obtained.   We must choose  $c(v_1')$ to be a color different from $c(v_{2})$, $c(v_{3})$, as well as two distinct colors in $N_G(w_1)-\{v'_1\}$, and also two distinct colors in $N_G(w'_1)-\{v'_1\}$.   Similarly, $c(v'_2)$ must be a color different from $c(v_1)$, $c(v_3)$, and at most four other colors from vertices outside of $S$, and $c(v'_3)$ must be a color different from $c(v_1)$, $c(v_2)$, and at most four other colors from vertices outside of $S$.

    We must also choose  $c(v_1)$ to differ from $c(v_2),c(v_3),c(v'_2),c(v'_3)$ and also different from $c(w_1)$ and $c(w'_1)$. Similarly, $c(v_2)$ must be different from $c(v_1),c(v_3),c(v'_1),c(v'_3)$ and also different from $c(w_2),c(w'_2)$, and $c(v_3)$ must be different from $c(v_1),c(v_2),c(v'_1),c(v'_2),c(w_3),c(w'_3)$.

For each vertex $u$ in $S$ let $R(u)$ be the set of those colors we need to avoid for $c(u)$ that come from vertices outside $S$. By the above argument we have $|R(u)|\leq 2$ when $u=v_i$ and $|R(u)|\leq 4$ when $u=v_i'$ for each $i$. For each vertex $u$ in $S$ define $L(u)=\{1,\ldots,6\}-R(u)$. 

Now we form a graph $D$ that represents the dependencies among the vertices of $S$. $D$ has vertex set $S$. Two vertices of $S$ are adjacent in $D$ if we require them to have different colors. 

First suppose that no pair of vertices in $\{v'_1,v'_2,v'_3\}$ have a common neighbor.  See Figure \ref{Fig-triangle-with-all-3-neighbors}.  In this case, in $D$ each $v_i$ has degree $4$ and each $v'_i$ has degree 2. Consider the list of colors $L(u)$ we defined on each vertex $u$ of $S$. Each vertex $u$ has a list of size at least its degree in $D$. Note that $D$ has one component  which is 2-connected and it is not an odd cycle or a complete graph.    Therefore by Theorem \ref{list} the graph $D$ is $L$-choosable. Such a coloring for the vertices of $S$ extends $c$ over $H-S$ to a 3-weak-dynamic coloring of $G$. 

If one or the three pair of vertices in $\{v'_1,v'_2,v'_3\}$ have common neighbors in $G$, then in $D$ we will have one, two, or the three edges $v'_1v'_2,v'_1v'_3,v'_2v'_3$ present, while still each vertex has a list of size at least its degree. Similar to the above argument, Theorem \ref{list} implies that $D$ is $L$-choosable, as desired.

\begin{figure}[H]
\begin{subfigure}{0.5\textwidth}
\centering
\begin{tikzpicture}[xscale = 0.75, yscale = 0.75]
\foreach \x in {1,...,3}
    \filldraw (120*\x -90:1.5) circle (3pt);
\foreach \x in {1,...,3}
    \filldraw (120*\x -90:2.5) circle (3pt);
\foreach \x in {1,...,3}
    \draw (120*\x -90:1.5) -- (120*\x + 30:1.5);
\foreach \x in {1,...,3}
    \draw (120*\x -90:1.5) -- (120*\x -90:2.5);
\foreach \x in {1,...,3}
    \draw (120*\x -100:3.5) -- (120*\x -90:2.5) -- (120*\x -80:3.5);
%\filldraw (1.3,.25) circle (1pt);
%\filldraw (1.3,0) circle (1pt);
%\filldraw (1.3,-.25) circle (1pt);
\node[anchor = west] at (-90:1.5) {$v_1$};
\node[anchor = north] at (30:1.5) {$v_2$};
\node[anchor = north] at (150:1.5) {$v_3$};
\node[anchor = west] at (-90:2.5) {$v'_1$};
\node[anchor = north] at (30:2.5) {$v'_2$};
\node[anchor = north] at (150:2.5) {$v'_3$};

\end{tikzpicture}
\caption{$G$}
\end{subfigure}
\begin{subfigure}{0.5\textwidth}
\centering
\begin{tikzpicture}[xscale = 0.75, yscale = 0.75]
\foreach \x in {1,...,3}
    \filldraw (120*\x -90:1.5) circle (3pt);
\foreach \x in {1,...,3}
    \filldraw (120*\x -90:2.5) circle (3pt);
\foreach \x in {1,...,3}
    \draw (120*\x -90:1.5) -- (120*\x + 30:1.5);
\foreach \x in {0,...,2}
    \draw (120*\x -90:1.5) -- (120*\x + 30:2.5);
\foreach \x in {1,...,3}
    \draw (120*\x -90:1.5) -- (120*\x - 210:2.5);

\node[anchor = west] at (-90:1.5) {$v_1$};
\node[anchor = north] at (30:1.5) {$v_2$};
\node[anchor = north] at (150:1.5) {$v_3$};
\node[anchor = west] at (-90:2.5) {$v'_1$};
\node[anchor = north] at (30:2.5) {$v'_2$};
\node[anchor = north] at (150:2.5) {$v'_3$};

\end{tikzpicture}
\caption{$D(S)$}
\end{subfigure}
\caption{A triangle with all 3-vertices.}\label{Fig-triangle-with-all-3-neighbors}
\end{figure}
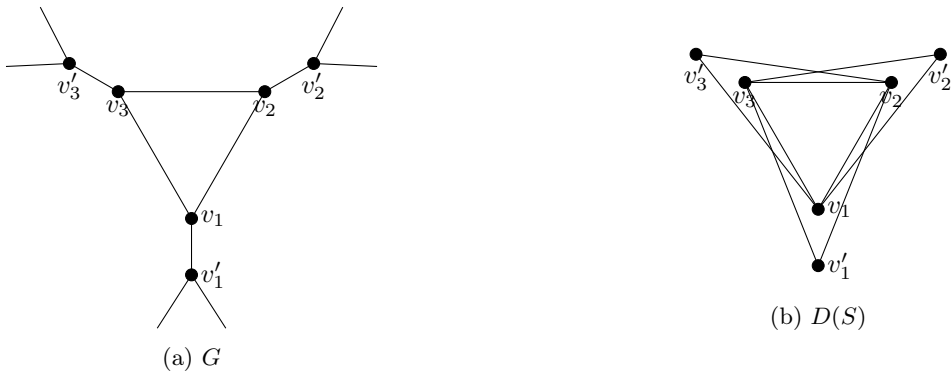

\item [Case 2:]  $d_G(v'_1)\geq 4$.  

Since $d_G(v'_1)\geq 4$, we have $d_H(v'_1)\geq 3$. Hence under the coloring $c$ in $H$, the vertex $v'_1$ sees at least three different colors on its neighborhood.  Therefore when trying to extend the coloring $c$ to a 3-weak-dynamic coloring of $G$, the vertex $v'_1$ is already satisfied.  In this case we keep the colors on all vertices of $H$. We then choose $c(v_1)$, $c(v_2)$, and $c(v_3)$ to extend $c$ to a 3-weak-dynamic coloring of $G$.

First choose $c(v_2)$ to be a color in $\{1,\ldots,6\}$ that is different from $c(v'_1)$, $c(v'_3)$, and different from two distinct  colors on vertices in  $N_G(v'_2)-\{v_2\}$. We then choose $c(v_3)$  to be a color in $\{1,\ldots,6\}$, different from $c(v_2)$ ,$c(v'_1)$, and $c(v'_2),$ and different from two distinct  colors on vertices in  $N_G(v'_3)-\{v_3\}$. Finally, considering the fact that $v'_1$ is already satisfied, we choose $c(v_1)$  to be a color in $\{1,\ldots,6\}$ and different from $c(v_2)$, $c(v_3)$, $c(v'_1)$, and $c(v'_2).$ It is easy to see that this extension provides a 3-weak-dynamic coloring of $G$, which is a contradiction.  
\end{description}
\end{proof}\qedhere

\begin{lemma}\label{lemma:edgeadjacentC_3}
The edge-minimal graph $G$ with $wd_3(G)>6$  contains no triangle with vertices $v_1,v_2,v_3$ adjacent to a triangle with vertices $v_1,v_3,v_4$ such that $d(v_2)=d(v_3)=d(v_4)=3$ and $d(v_1)\ge 4$. 
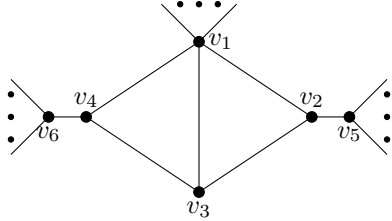
\begin{figure}[H]
\centering
\begin{tikzpicture}
\draw (0,0) -- (-1.5,-1) -- (0,-2) -- cycle -- (1.5,-1) -- (0,-2)
(-1.5,-1) -- (-2,-1) (1.5,-1) -- (2,-1)
(-2.5,-.5) -- (-2,-1) -- (-2.5,-1.5) (2.5,-.5) -- (2,-1) -- (2.5,-1.5)
(-.5,.5) -- (0,0) -- (.5,.5);
\filldraw (0,0) circle (2pt) node[anchor=west]{$v_1$};
\filldraw (1.5,-1) circle (2pt) node[anchor=south]{$v_2$};
\filldraw (0,-2) circle (2pt) node[anchor=north]{$v_3$};
\filldraw (-1.5,-1) circle (2pt) node[anchor=south]{$v_4$};
\filldraw (2,-1) circle (2pt) node[anchor=north]{$v_5$};
\filldraw (-2,-1) circle (2pt) node[anchor=north]{$v_6$};
\filldraw (-2.5,-1.3) circle (1pt);
\filldraw (-2.5,-1) circle (1pt);
\filldraw (-2.5,-.7) circle (1pt);
\filldraw (2.5,-.7) circle (1pt);
\filldraw (2.5,-1) circle (1pt);
\filldraw (2.5,-1.3) circle (1pt);
\filldraw (.25,.5) circle (1pt);
\filldraw (-.25,.5) circle (1pt);
\filldraw (0,.5) circle (1pt);
\end{tikzpicture}
\caption{Two adjacent triangles.}\label{Fig-Two adjacent triangles'}
\end{figure}
\end{lemma}
\begin{proof}
On the contrary suppose $G$ contains this configuration. Let $N_G(v_2)=\{v_1,v_3,v_5\}$ and $N_G(v_4)=\{v_1,v_3,v_6\}$.  Let $H=G-\{v_3\}$. Since $H$ has fewer edges than $G$, we have $wd_3(H)\leq 6$.  Therefore, there exists $c:V(H)\rightarrow \{1,\ldots,6\}$ that is a 3-weak-dynamic coloring of $H$.  To find a 3-weak-dynamic coloring of $G$, we recolor vertices $v_2$ and $v_4$ and find an appropriate color for $v_3$.

Let $c_1$ and $c_2$ be two different colors on $N_H(v_5)-\{v_2\}$, let $c_3$ and $c_4$ be two different colors on  $N_H(v_6)-\{v_4\}$, and let $c_5$ be a color on $N_H(v_1)-\{v_2,v_4\}$. 

 We first recolor $v_2$ to be a color in $\{1,\ldots,6\}-\{c(v_1),c_1,c_2,c_5\}$. Now choose $c(v_3)$ to be a color in $\{1,\ldots,6\}-\{c(v_1),c(v_2),c(v_5),c(v_6),c_5\}$. Note that the vertex $v_1$ becomes satisfied at this stage. Finally recolor $v_4$ to be a color in     $\{1,\ldots,6\}-\{c(v_1),c(v_2),c_3,c_4\}$. Since each of the vertices $v_1,v_2,v_3,v_4,v_5,$ and $v_6$ become satisfied with these assignments of colors and since $c$ satisfies all other vertices of $H$, we obtain a 3-weak-dynamic coloring of $G$.
\end {proof}

\begin{lemma}\label{lemma:C_3withDegree3Neighbors}
The edge-minimal graph $G$ does not contain a triangle with vertices $v_1,v_2,v_3$, where $d(v_1)=d(v_2)=3$ and $d(v_3)=4$ such that each of $v_1$ and $v_2$ has only one $4^+$-neighbor.
\begin{figure} [H]
\centering
\begin{tikzpicture}[xscale = 0.75, yscale = 0.75]
\foreach \x in {1,...,3}
    \filldraw (120*\x-30:1.5) circle (3pt);
\foreach \x in {1,...,3}
 \draw (120*\x-30:1.5) -- (120*\x+90:1.5);
\draw (210:1.5) -- ++(270:1.25) node(x) {}++(240:1) -- ++(60:1) -- ++(300:1);   
\draw (330:1.5) -- ++(270:1.25) node(y) {} ++(240:1) -- ++(60:1) -- ++(300:1);
\draw (330:1.5) -- ++(270:1.25) node(z) {} ++(240:1) -- ++(60:1) -- ++(300:1);
\draw (110:3.5) -- (105:2.5) node[anchor = east] {$v_6$} -- (90:1.5) -- (75:2.5) node[anchor = west] {$v_7$} -- (70:3.5);
%\draw (105:3.5) -- (105:2.5) -- (100:3.5);% -- (90:1.5) -- (75:2.5) node[anchor = west] {$v_7$};
\filldraw (94:3.35) -- (105:2.5);
\filldraw (86:3.35) -- (75:2.5);
\filldraw (105:2.5) circle (3pt);
\filldraw (75:2.5) circle (3pt);
\filldraw (x) circle (3pt) node[anchor = east] {$v_4$};
\filldraw (y) circle (3pt) node[anchor = west] {$v_5$};
%\filldraw (z) circle (3pt) node[anchor = east] {$v_8$} ++(270:0.75) circle (1pt) ++(0:0.25) circle(1pt) ++(180:0.5) circle (1pt);
%\filldraw (w) circle (3pt) node[anchor = west] {$v_9$} ++(270:0.75) circle (1pt) ++(0:0.25) circle(1pt) ++(180:0.5) circle (1pt);
\node[anchor = east] at (210:1.5) {$v_1$};
\node[anchor = west] at (330:1.5) {$v_2$};
\node[anchor = west] at (90:1.5) {$v_3$};
\end{tikzpicture}
\caption{A triangle with a vertex of degree 4.}\label{Fig-A-triangle-with-vertex-degree4}
\end{figure}
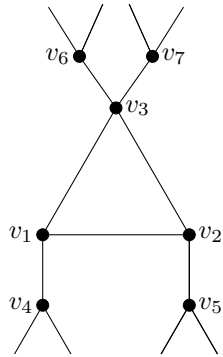
\end{lemma}
\begin{proof}
On the contrary suppose, $G$ contains this configuration. Let $N_G(v_1)-\{v_2,v_3\}=\{v_4\}$,   $N_G(v_2)-\{v_1,v_3\}=\{v_5\}$, and $N_G(v_3)-\{v_1,v_2\}=\{v_6,v_7\}$. Since each of $v_1$ and $v_2$ has only one $4^+$-neighbor, Lemma \ref{lemma:AdjacentDegree2s} implies that $d_G(v_4)=d_G(v_5)=3$.  Moreover Lemma \ref{4-4} implies that  $d_G(v_6)\leq 3$ and $d_G(v_7)\leq 3$. We may suppose that $d_G(v_6)=d_G(v_7)= 3$, because degree 3 vertices provide more restrictions on the coloring.

 Contract the edge $v_1v_2$ to a single vertex $v_{1,2}$ and let $H$ be the resulting graph. Since $H$ has fewer edges than $G$, we have $wd_3(H) \le 6$. Therefore there is $c:V(H) \to \{1,...,6\}$ that is a   3-weak-dynamic coloring of $H$. We aim to reach a contradiction by using $c$ to extend the coloring of $H$ to $G$.  Let $c_1$ and $c_2$ be two distinct colors  in $c(N_H(v_4)-\{v_{1,2}\})$, and let $c_3$ and $c_4$ be two distinct colors  in $c(N_H(v_5)-\{v_{1,2}\})$. Note that $c(v_6)\neq c(v_7)$, because $v_3$ has degree 3 in $H$.

 We consider three cases.
\begin{description}
\item [Case 1:] $|\{c_1,c_2,c(v_6),c(v_7),c(v_3),c(v_5)\}| < 6$.

 In this case, we  keep the coloring of $c$ over all vertices of $V(H)-\{v_{1,2}\}$.   Choose $c(v_1)$ to be a color in $\{1,\ldots,6\}-\{c_1,c_2,c(v_6),c(v_7),c(v_3),c(v_5)\}$ that satisfies $v_2$, $v_3$, and $v_4$. Then assign $v_2$ a color in $\{1,\ldots,6\}-\{c_3,c_4,c(v_3),c(v_4)\}$ that satisfy $v_1$ and $v_5$. Therefore we obtain a 3-weak-dynamic coloring of $G$ with at most six colors.

\item [Case 2:] $|\{c_3,c_4,c(v_6),c(v_7),c(v_3),c(v_4)\}| <6$. 

 In this case, we  keep the coloring of $c$ over all vertices of $V(H)-\{v_{1,2}\}$.   Choose $c(v_2)$ to be a color in $\{1,\ldots,6\}-\{c_3,c_4,c(v_6),c(v_7),c(v_3),c(v_4)\}$, satisfying $v_1$, $v_3$, and $v_5$. Then  assign $v_1$ a color in $\{1,\ldots,6\}-\{c_1,c_2,c(v_3),c(v_5)\}$ to satisfy $v_2$ and $v_4$. Therefore we obtain a 3-weak-dynamic coloring of $G$ with at most six colors. 
 
\item [Case 3:] $\{c_1,c_2,c(v_6),c(v_7),c(v_3),c(v_5)\}=\{c_3,c_4,c(v_6),c(v_7),c(v_3),c(v_4)\}=\{1,\ldots,6\}$.

Therefore we have $\{c_1,c_2,c(v_5)\}=\{c_3,c_4,c(v_4)\}$. Since $v_4$ and $v_5$ have a common 3-neighbor in $H$, we have $c(v_4)\neq c(v_5)$. Hence we may suppose that $c(v_4)=c_1$, $c(v_5)=c_3$, and $c_2=c_4$. As a result, we may suppose that $c_1=c(v_4)=1$, $c_2=c_4=2$, $c_3=c(v_5)=3$, $c(v_3)=4$, $c(v_6)=5$, and $c(v_7)=6$.

Let $N_G(v_4)=\{v_8,v_9\}$ and let $N_G(v_5)=\{v_{10},v_{11}\}$. Let $c_7$ and $c_8$ be two distinct colors on the neighborhood of $v_8$, and let $c_9$ and $c_{10}$ be two distinct colors on the neighborhood of $v_9$.  Now recolor $v_4$ to be a color in $\{1,\ldots,6\}$ different from its current color (color 1) and different from  $\{c_7,c_8,c_9,c_{10}\}$. If the new color of $v_4$ is not $4$, then choose $c(v_2)$ to be equal to $1$ to satisfy $v_1,v_3,v_5$.  Then  assign $v_1$ a color in $\{1,\ldots,6\}-\{1,2,3,4\}$  to satisfy $v_2$ and $v_4$. Therefore we obtain a 3-weak-dynamic coloring of $G$ with at most six colors. 

Hence we may suppose we have recolored $v_4$ and the new color is $4$, i.e.\ $c(v_4)=4$. By a similar argument as above, we may also recolor $v_5$ and we can suppose that the new color on $v_5$ is $4$ too.  Now recolor $v_3$ to be a color different from 4,  different from two distinct colors in $c(N_G(v_6)-\{v_3\})$, and different from two distinct colors in $c(N_G(v_7)-\{v_3\})$. Now consider the new coloring on $v_3,v_4,$ and $v_5$. 

 If $c(v_3)\neq 3$, then let $c(v_1)=3$ and choose $c(v_2)$ to be a color in $\{1,5,6\}-\{c(v_3)\}$.  If $c(v_3)=3$, then let $c(v_1)=5$ and  $c(v_2)=1$. In the both cases, $c$ provides a 3-weak-dynamic coloring of $G$, which is a contradiction.
\end{description}
\end{proof}
%%%
\begin{lemma}\label{lemma:C_3withDegree5top}
The edge minimal graph $G$ does not contain a triangle with vertices $v_1,v_2,v_3$, such that $d(v_1)=d(v_2)=3$, $d(v_3)\geq 5$, and each of $v_1$ and $v_2$ has only one $4^+$-neighbor.
\end{lemma}

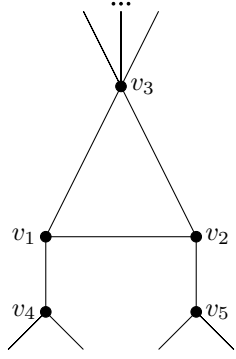
\begin{figure} %[H]
\centering
\begin{tikzpicture}
\draw (-1,0)--(0,2)--(1,0)--(-1,0);
\draw (-1,0)--(-1,-1)--(-1.5,-1.5)--(-1,-1)--(-.5,-1.5);
\draw (1,0)--(1,-1)--(1.5,-1.5)--(1,-1)--(.5,-1.5);
\draw (0,2)--(-.5,3)--(0,2)--(0,3)--(0,2)--(.5,3);
\filldraw (-1,0) circle (2pt)  node[anchor=east]{$v_1$};
\filldraw (1,0) circle (2pt) node[anchor=west]{$v_2$};
\filldraw (0,2) circle (2pt) node[anchor=west]{$v_3$};
\filldraw (-1,-1) circle (2pt) node[anchor=east]{$v_4$};
\filldraw (1,-1) circle (2pt) node[anchor=west]{$v_5$};
%\node[anchor=east] at (-.5,3) {$v_6$};
%\node[anchor=west] at (.5,3) {$v_8$};
%\node[anchor=east] at (-.25,3.25) {$v_7$};
\filldraw (0,3.1) circle (.5pt);
\filldraw (.1,3.1) circle (.5pt);
\filldraw (-.1,3.1) circle (.5pt);
\end{tikzpicture}
\caption{A triangle with a vertex of degree at least 5.}\label{Fig-A-triangle-with-vertex-degree5}
\end{figure}

\begin{proof}
On the contrary suppose $G$ contains this configuration. Let $H=G-\{v_1,v_2\}$. Since $H$ has fewer edges than $G$, we have $wd_3(H) \le 6$. Therefore there exists $c: V(H) \to \{1,...,6\}$ that is a 3-weak-dynamic coloring of $H$. Let $N_G(v_1)=\{v_2,v_3,v_4\}$ and $N_G(v_2)=\{v_1,v_3,v_5\}$. Since $d_G(v_4)\leq 3$ and $d_G(v_5)\leq 3$, Lemma \ref{lemma:AdjacentDegree2s} implies that $d(v_4)=d(v_5)=3$.  Fix the coloring $c$ over the vertices $V(G)-\{v_1,v_2,v_4,v_5\}$. We recolor $v_4$ and $v_5$ and then find appropriate colors for $v_1$ and $v_2$ to obtain a 3-weak-dynamic coloring of $G$.   Note that $v_3$ was satisfied by the coloring of $H$ since $d_H(v_3) \ge 3$. Therefore, when we color $v_1$ and $v_2$,  the neighbors of $v_3$ do not create any dependencies for them. 

We begin by recoloring $v_4$ and $v_5$. We have $d(v_4)=d(v_5)=3$ and therefore, by the coloring of $H$, we know that $v_4$ must avoid two colors from the neighborhood of each vertex in $N(v_4)-\{v_1\}$. Additionally $v_4$ must avoid $c(v_3)$. Therefore we have only  five dependencies on $v_4$ and we are able to choose an appropriate color for $v_4$ in $\{1,\ldots,6\}$.  Similarly we have that $v_5$ must avoid at most five colors. Therefore we can recolor $v_5$ as well.

 Now choose $c(v_1)$  to be a color in $\{1,\ldots,6\}$, different from $c(v_3)$ and $c(v_5)$, and also different from the colors of the two vertices in $N_G(v_4)-\{v_1\}$.   Finally choose $c(v_2)$  to be a color in $\{1,\ldots,6\}$, different from $c(v_3)$  and $c(v_4)$, and also different from the colors of the two vertices in $N_G(v_5)-\{v_2\}$.  This new coloring is a 3-weak-dynamic coloring of $G$, a contradiction.
\end{proof}

%
%\begin{lemma}\label{lemma:3-Regulark-cycle-with-all-3-neighbors}
%The edge-minimal graph $G$ with $wd_3(G)>6$ contains no cycle $C$ with vertices $v_1,\ldots,v_k$ such that $d(v_1)=\ldots=d(v_k)=3$,  when $k$ is odd at least a vertex in $\{v_1,\ldots,v_k\}$ has no $4^+$-neighbor, and when $k$ is even at least a vertex in $\{v_1,v_3,\ldots,v_{k-1}\}$ and at least a vertex in $\{v_2,v_4,\ldots,v_k\}$ has no $4^+$-neighbor.
%\end{lemma}

\begin{lemma}\label{lemma:3-Regulark-cycle-with-all-3-neighbors}
The edge-minimal graph $G$ with $wd_3(G)>6$ contains no cycle $C$ with vertices $v_1,\ldots,v_k$ such that $d(v_1)=\ldots=d(v_k)=3$, and
\begin{enumerate}
\item when $k$ is odd, a vertex in $\{v_1,\ldots,v_k\}$ has no $4^+$-neighbor, and
\item  when $k$ is even, a vertex in $\{v_1,v_3,\ldots,v_{k-1}\}$ and a vertex in $\{v_2,v_4,\ldots,v_k\}$ both have no $4^+$-neighbor.
\end{enumerate} 
\end{lemma}

\begin{proof}
On the contrary, suppose $G$ contains such a configuration $C$. We may choose $C$ to be the shortest such configuration.  Hence $C$ has no chord. For each $i$, let $v'_i$ be the neighbor of $v_i$ outside $C$. Note  $v'_1,\ldots, v'_k$  are not necessarily  distinct vertices, but they are distinct from $v_1,\ldots,v_k$ because $C$ has no chord. Let $H=G-\{v_1,\ldots,v_k\}$.  Since $H$ has fewer edges than $G$, we have $wd_3(H)\leq6$.  Thus  there exists $c:V(H)\rightarrow\{1,\ldots,6\}$ that is a 3-weak-dynamic coloring of $H$.  To obtain a contradiction, we use $c$ to find a 3-weak-dynamic coloring of $G$.

By Lemma \ref{lemma:AdjacentDegree2s} all the vertices $v'_1,\ldots, v'_k$ have degree at least 3 in $G$. By the structure of $C$, not all vertices in $\{v'_1,\ldots,v'_{k}\}$ have degree at least 4. Hence  we may suppose that when $k$ is odd, $d(v'_1)=3$, and when $k$ is even, $d(v'_1)=d(v'_2)=3$. The proof of the remaining cases is very similar.

%We recolor $v_1',\ldots,v_k'$ and find appropriate colors for $v_1,\ldots,v_k$.  We will call the set of vertices that we want to colore or recolore $S$.  Thus  $S=\{v_1,\ldots,v_k,v_1',\ldots,v_k'\}$. 

Let $S=\{v_1,\ldots,v_k\}$. We aim to extend the coloring $c$ to a 3-weak-dynamic coloring of $G$ by choosing appropriate colors for the vertices in $S$. Now we study the restrictions we must consider  for the coloring on $S$ to make sure that a 3-weak-dynamic coloring of $G$ is obtained.  Let $i\in\{1,\ldots,k\}$. If $v'_i$ appears only once in the multiset $\{v'_1,\ldots,v'_k\}$, then we choose   $c(v_i)$ to be different from $c(v_{i+2}),c(v_{i-2}),c(v_{i+1}'),c(v_{i-1}')$ as well as at most two distinct colors in $N_H(v'_i)$.

  If $v'_i$ appears twice in the multiset $\{v'_1,\ldots,v'_k\}$, then in $G$ the vertex $v'_i$ is adjacent to two vertices of $C$. As a result we choose the color of $v_i$ to be different from a color in $N_H(v'_i)$ and different from $c(v_{i+2}),c(v_{i-2}),c(v_{i+1}'),c(v_{i-1}')$, and different from the color of an additional  vertex in $C$ (the vertex $v_j$ such that $v'_i=v'_j$).

For any vertex $x$ that appears at least three times in the multiset $\{v'_1,\ldots,v'_k\}$, choose $S_x$ to consist of three indices $j_1,j_2,j_3$ such that $x=v'_{j_1}=v'_{j_2}=v'_{j_3}$. Then if we choose the colors of the vertices $v_{j_1},v_{j_2},v_{j_3}$ to be different, the vertex $x$ becomes satisfied in $G$. Therefore  if $v'_i$ appears three or more  times in the multiset $\{v'_1,\ldots,v'_k\}$, then we choose  the color of $v_i$   to be different from  $c(v_{i+1}),c(v_{i-2}),c(v'_{i+1}),c(v'_{i-1})$ and moreover if $i\in S_{v'_i} $ choose $c(v_i)$ to be also different from the color  of two other vertices in $C$ (the two vertices other than $v_i$ whose indices belong to $S_{v'_i}$). Note that by the way we aim to choose colors for the vertices $v_1,\ldots,v_k$, if this extension exists, all the vertices $v_1,\ldots,v_k,v'_1,\ldots,v'_k$ become satisfied.

Now we form a graph $D$ that represents the dependencies among the vertices of $S$.  The graph $D$ has vertex set $S$, and two vertices of $S$ are adjacent in $D$ if we require their colors to be different.
For each vertex $w$ in $S$, let $R(w)$ be the set of those colors we need to avoid for $c(w)$ that come from vertices outside of  $S$.   Define $L(w)=\{1,\ldots,6\}-R(w)$. By the above argument each vertex of $S$ has at most six restrictions, hence  $|L(w)|$ is at least the degree of $w$ in $D$ for all $w\in S$. It is enough to show that $D$ is $L$-choosable, because then the coloring of vertices of $D$ can be used on the corresponding vertices in $G$ to extend $c$ to a 3-weak-dynamic coloring of $G$. 

In $D$ each vertex $v_i$ is adjacent to $v_{i-2}$, $v_{i+2}$. When $v'_i$ appears more than once in the multiset $\{v'_1,\ldots,v'_k\}$, the vertex $v_i$ might have other neighbors in $D$ as well. As a result when $k$  is odd, $D$ has one component which is Hamiltonian, and when $k$ is even, $D$ has  at most two components. %Therefore none of the components of $D$ is a cycle.

% Moreover all components of $D$ are 2-connected.  By Lemma \ref{lemma:C_3withDegree3Neighbors} we have $k\geq 4$.

%Note that according to the adjacencies between vertices of $D$, when $k$ is odd $D$ has only one component and that components contains a Hamiltonian cycle. When $k$ is even, $D$ contains at most two components and if it has exactly two components, the vertices $v_1$ and $v_2$ belong to different components. 

 By Lemma \ref{3-uniform triangle}, we have $k\neq 3$.   When $k=4$ each of the vertices $v_1,\ldots,v_4$ has at most five restrictions, which makes their lists larger than their degrees.  By Corollary~\ref{cor-list}, $D$ is $L$-choosable in this case. Hence suppose $k\geq 5$. 

First suppose that $D$ is 2-connected.  If $D$ is not a complete graph, an odd cycle, if $D$  has a vertex $u$ with $|L(u)|>d_D(u)$, or if not all vertices of $D$ have the same lists, then by Theorem~\ref{list}, Corollary~\ref{cor-list}, Proposition~\ref{complete-graph}, and Proposition~\ref{odd-cycle} the graph $D$ is $L$-choosable, as desired. Hence suppose $D$ is an odd cycle or a complete graph,  all its lists are the same, and have size equal to the degrees of the vertices in $D$. Recall that vertex $v'_1$ has degree 3 in $G$. Thus the degree of $v'_1$ in $H$ is at most $2$. Therefore we can recolor $v'_1$ in $H$ by another color in such a way that the coloring on $H$ stays 3-weak-dynamic. Let $c^*$ be the new 3-weak-dynamic coloring of $H$. Now repeat the above argument over the coloring $c^*$ of $H$. 

Since $|L(v_i)|=d_D(v_i)$ for all $i$, we have $v'_{i+1}\neq v'_{i+1}$ (otherwise $v_i$ has at most five restrictions).  Moreover the choice of $C$ and Lemma \ref{3-uniform triangle} imply  that $v'_1$ appears at most once in the multiset $\{v'_1,\ldots,v'_k\}$. Hence by moving from the coloring $c$ to the coloring $c^*$, the lists of the vertices $v_2$ and $v_k$ change to another list, while the lists on other vertices stay as before. Therefore not all the lists are the same now. As a result, by Corollary  \ref{cor-list} and Propositions \ref{complete-graph} and \ref{odd-cycle}, the graph $D$ is $L$-choosable, as desired.

Recall that when $k$ is odd, $D$ is Hamiltonian. Hence for the case that $k$ is odd, or $k$ is even but $D$ is 2-connected, the above argument shows that $D$ is $L$-choosable. Now suppose that $k$ is even and $D$ is not 2-connected. The graph $D$ contains at most two components. 

If $D$ has exactly two components $C_1$ and $C_2$, then vertices $v_1$ and $v_2$ belong to different components of $D$, because we know that $v_1v_3\ldots v_{k-1}v_1$ and $v_2v_4\ldots v_kv_2$ are cycles in $D$. Moreover each of the components is 2-connected, because they are Himiltonian. Since $v'_1$ and $v'_2$ have degree at most $2$ in $H$, a similar argument as the one we applied above can be applied here independently for $C_1$ and $C_2$ to extend the coloring $c$ (and change it if necessary) to a 3-weak-dynamic coloring of $G$. 

 Hence suppose $D$ is connected, but is not 2-connected. Therefore $D$ has two blocks, one with vertices of odd indices, say $B_1$, and one with vertices of even indices, say $B_2$. Therefore $D$ has a cut-vertex $v$. We may suppose that $v$ belongs to $B_1$.

Now choose colors for vertices of $B_2$ from their lists in such a way that a proper coloring for $B_2$ is obtained. This is possible because all vertices of $B_2$ have lists of size at least their degrees and at least one vertex of $B_2$ (the neighbor(s) of $v$ in $B_2$) has a list of size one more than its degree in $B_2$. Note that $v$ is the only vertex of $B_1$ that has a neighbor in $B_2$, since otherwise $v$ cannot be a cut-vertex of $D$. Now redefine $L(v)$ by removing from it the colors that are already picked for the neighbor(s) of $v$ in $B_2$.  Now consider the new list assignment $L$ over the vertices of $B_1$. Each vertex has a list of size at least its degree in $B_1$, and $B_1$ is 2-connected. If $B_1$ is not a complete graph or odd cycle (Theorem \ref{list}),  if $B_1$ is a complete graph or odd cycle but the lists on its vertices are not identical (Corollary \ref{cor-list}), or if  $B_1$ is a complete graph or odd cycle but it has a vertex $u$ with $|L(u)|>d_{B_1}(u)$ (Propositions \ref{complete-graph} and \ref{odd-cycle}), then $B_1$ is $L$-choosable, as desired.

Hence suppose $B_1$ is a complete graph or odd cycle, and the lists on the vertices of $B_1$ are identical and have size equal to the degrees of vertices in $B_1$.  Recall that we supposed $d_G(v'_2)=3$. Hence in $H$ the vertex $v'_2$ has degree at most 2. Therefore we can recolor this vertex using a color in $\{1,\ldots,6\}$ by a different color in such a way that the new coloring $c^*$ is still a 3-weak-dynamic coloring of $H$. Now repeat the same process as above on  defining a list $L'$ on the vertices of $D$, but using coloring $c^*$ in place of color $c$.

The vertex $v'_2$ appears only once in the multiset $\{v'_1,\ldots,v'_k\}$, because if $v'_2=v'_4$ or $v'_2=v'_k$, then the vertex $v_3$ or the vertex $v_{k-1}$ have lists of size larger than their degrees in $D$, which is not accepted. If $v'_2=v'_j$ for some $j\not\in \{4,k-1\}$, then a configuration smaller than $C$ exists in $G$, which is also not accepted by the choice of $C$.

Note that the only difference between colorings $c$ and $c^*$ is on the color of vertex $v'_2$. By the argument in the above paragraph, only the list of vertices $v_1$ and $v_3$ are affected by the color of the vertex $v'_2$.  Hence the only difference between $L$ and $L'$ is on the lists of vertices $v_1$ and $v_3$. Therefore the vertices of $B_2$ get the same colors as before, because for these vertices $L$ and $L'$ are the same. Now redefine $L'(v)$ by removing from it the color of  neighbors of $v$ in $B_2$. Now we try to color the vertices of $B_1$ using the list assignment $L'$.  But exactly two vertices of $B_1$ (the vertices $v_1$ and $v_3$) have different lists than before. Moreover $k\geq 5$ implies that $B_1$ has at least three vertices. Therefore not all lists on the vertices of $B_1$ are now the same. Hence by Corollary \ref{cor-list}, Proposition \ref{complete-graph}, and Proposition \ref{odd-cycle},   $B_1$ is $L'$-choosable, as desired.
\end{proof}

%%%

\begin{lemma}\label{lemma:3-RegularkCycle-remaining-cases}
The edge-minimal graph $G$ with $wd_3(G)>6$ contains no cycle $C$ with vertices $v_1,\ldots,v_{k}$ such that $d(v_1)=\ldots=d(v_k)=3$.
\end{lemma}

\begin{proof}
On the contrary suppose $G$ contains such a configuration $C$. We may choose $C$ to be the shortest cycle in $G$ that forms this configuration. Therefore $C$ has no chord.  For each $i$, let $v'_i$ be the neighbor of $v_i$ outside $C$. Hence, while  $v'_1,\ldots v'_k$  are not necessarily  distinct vertices, by the choice of $C$ they are distinct from $v_1,\ldots,v_k$.   By Lemmas \ref{3-uniform triangle},  \ref{lemma:C_3withDegree5top}, and  \ref{lemma:3-Regulark-cycle-with-all-3-neighbors}, we have $v'_i\neq v'_{i+1}$ for all $i$. By Lemma \ref{lemma:4-3-3-4Configuration}, $d(v'_i)\geq 4$ and $d(v'_{i+1})\geq 4$ do not simultaneously happen for all $i$. Therefore by Lemma \ref{lemma:3-Regulark-cycle-with-all-3-neighbors},  $k$ is even. Moreover by Lemma \ref{lemma:3-Regulark-cycle-with-all-3-neighbors}, all vertices in $\{v'_1,v'_3,\ldots,v'_{k-1}\}$ or all vertices in $\{v'_2,v'_4,\ldots,v'_{k}\}$ have degree at least 4 in $G$. By symmetry, suppose all vertices in $\{v'_1,v'_3,\ldots,v'_{k-1}\}$ have degree at least 4 in $G$. As a result by Lemmas \ref{lemma:AdjacentDegree2s} and \ref{lemma:4-3-3-4Configuration},  all vertices in $\{v'_2,v'_4,\ldots,v'_{k}\}$ have degree 3 in $G$.

 Let $H=G-\{v_1,\ldots,v_k\}$.  Let $H'$ be the graph obtained from $H$ by identifying vertices $v'_1$ and $v'_{3}$ in $H$ into a single vertex $v'_{1,3}$. Note that $H'$ is still planar and has fewer edges than $G$. Therefore we have $wd_3(H')\leq6$.  Thus there exists $c:V(H')\rightarrow\{1,\ldots,6\}$ that is a 3-weak-dynamic coloring of $H$. Now give each vertex $v$ in $H$ the color its corresponding vertex in $H'$ has. Also give vertices $v'_1$ and $v'_3$ in $H$ the color of the vertex $v'_{1,3}$ in $H'$. In the current coloring of $H$ all the vertices of $H$ are satisfied (with respect to 3-weak-dynamic coloring property) except for possibly vertices $v'_1$ and $v'_3$. 
 
 If $v'_1$ sees only one color on its neighborhood in $H$, then choose a neighbor $x$ of $v'_1$ (which we know has degree at most 3 by Lemma \ref{lemma:AdjacentDegree2s}). We can recolor $x$ by   a different color in $\{1,\ldots,6\}$ in such a way that its neighbors in $N_H(x)-\{v'_1,v'_3\}$ stay satisfied. Similarly, we can recolor a neighbor of $v'_3$ in $H$, when $v'_3$ sees only one color on its neighborhood in $H$. Let $c^*$ be the resulting coloring on $H$.   We extend $c^*$ to a 3-weak-dynamic coloring of $G$ by finding  appropriate colors for $v_1,\ldots,v_k$.  We will call the set of vertices that we want to color $S$.  Thus, $S=\{v_1,\ldots,v_k\}$.  Now we study the restrictions we must consider  for the coloring on $S$ to make sure that a 3-weak-dynamic coloring of $G$ is obtained.

For each odd $i$ with $i\not\in \{1,3\}$, if $v'_i$ appears only once in the multiset $\{v'_1,\ldots,v'_k\}$, then $v'_i$ is already satisfied in $H$. Therefore it is enough to choose  $c(v_i)$ to be different from $c(v_{i+2})$, $c(v_{i-2})$, $c(v_{i+1}')$, and $c(v_{i-1}')$. For such an $i$, if $v'_i$ appears twice in $\{v'_1,\ldots,v'_k\}$, then we choose $c(v_i)$ to be different from $c(v_{i+2})$, $c(v_{i-2})$, $c(v_{i+1}')$, $c(v_{i-1}')$, and different form two colors in $N_H(v'_i)$.

For any vertex $x$ that appears at least three times in the multiset $\{v'_1,\ldots,v'_k\}$, choose $S_x$ to be a set containing three indices $j_1,j_2,j_3$ such that $x=v'_{j_1}=v'_{j_2}=v'_{j_3}$. Thus if we choose the colors of the vertices $v_{j_1},v_{j_2},v_{j_3}$ to be different, the vertex $x$ becomes satisfied in $G$. Therefore, for the case that $i$ is odd and $i\not\in\{1,3\}$,  if $v'_i$ appears three or more  times in the multiset $\{v'_1,\ldots,v'_k\}$, then we choose  the color of $v_i$   to be different from  $c(v_{i+2})$, $c(v_{i-2})$, $c(v'_{i+1})$, and $c(v'_{i-1})$.
If moreover $i\in S_{v'_i}$, then choose $c(v_i)$ to be different from $c(v_{i+2})$, $c(v_{i-2})$, $c(v'_{i+1})$, and $c(v'_{i-1})$ and different from the color  of two other vertices in $C$ (the two vertices other than $v_i$ whose indices belong to $S_{v'_i}$). 

Now suppose $i\in\{1,3\}$. Note that the vertices $v'_1$ and $v'_3$ might not be satisfied in $H$.  If $v'_i$ appears only once in $\{v'_1,\ldots,v'_k\}$, then choose  $c(v_i)$ to be different from $c(v_{i+2})$, $c(v_{i-2})$, $c(v_{i+1}')$, $c(v_{i-1}')$, and also different from two colors in $N_H(v'_i)$. %For such an $i$ 
If $v'_i$ appears twice in $\{v'_1,\ldots,v'_k\}$, then we choose $c(v_i)$ to be different from $c(v_{i+2})$, $c(v_{i-2})$, $c(v_{i+1}')$, $c(v_{i-1}')$, and different form two colors in $N_H(v'_i)$. And if  $v'_i$ appears three or more  times in the multiset $\{v'_1,\ldots,v'_k\}$, then we choose  the color of $v_i$ to be different from  $c(v_{i+2})$, $c(v_{i-2})$, $c(v'_{i+1})$, $c(v'_{i-1})$ and when $i\in S_{v'_i} $ choose $c(v_i)$ to be also different from the color of two other vertices in $C$ (the two vertices other than $v_i$ whose indices belong to $S_{v'_i}$).

For each even $i$, the vertex $v'_i$ appears at most twice in the multiset $\{v'_1,\ldots,v'_k\}$, since otherwise a configuration smaller than $C$ exists in $G$. In fact when $k\neq 4$, the vertex $v'_i$ appears at most once in the multiset $\{v'_1,\ldots,v'_k\}$, by the same reason. If $v'_i$ appears only once in $\{v'_1,\ldots,v'_k\}$, then choose  $c(v_i)$ to be different from $c(v_{i+2})$, $c(v_{i-2})$, $c(v_{i+1}')$ ,$c(v_{i-1}')$, and also different from two colors in $N_H(v'_i)$. %For such an $i$ 
If $v'_i$ appears twice in $\{v'_1,\ldots,v'_k\}$, i.e.\ if $k=4$ and $v'_2=v'_4$, then we choose $c(v_i)$ to be different from $c(v_{i+2})$, $c(v_{i-2})$, $c(v_{i+1}')$, $c(v_{i-1}')$, and different from the color of the vertex in $N_H(v'_i)$.

Now we form a graph $D$ that represents the dependencies among the vertices of $S$.  The graph $D$ has vertex set $S$ and two vertices of $S$ are adjacent in $D$ if we require their colors to be different.  For each vertex $w$ in $S$, let $R(w)$ be the set of those colors we need to avoid for $c(w)$ that come from vertices outside $S$.   Define $L(w)=\{1,\ldots,6\}-R(w)$. By the above argument, each vertex of $S$ has a total of at most six restrictions. Moreover vertices of indices in $\{5,7,\ldots,k-1\}$ have four restrictions. Since $c^*(v'_1)=c^*(v'_3)$, the vertex $v_2$ has at most five restrictions, and finally when $k=4$, all the vertices of $S$ have at most five restrictions, because $v_{i+2}$ and $v_{i-2}$ are the same vertices in this case. 

 Hence  $|L(w)|$ is at least the degree of $w$ in $D$ for all $w\in S$, and $|L(w)|$ has size more than the degree of $w$ in $D$ when $w\in \{v_2,v_5,v_7,\ldots,v_{k-1}\}$. Therefore it is enough to show that $D$ is $L$-choosable, because in this case the proper coloring we obtain for $D$ would be an extension of $c^*$ to a 3-weak-dynamic coloring of $G$.
 
 Recall that $k$ is even. If $k=4$, then since the lists on all vertices have size larger than their degrees in $D$ the graph $D$ is $L$-choosable  by Corollary \ref{cor-list}. Thus suppose $k\geq 6$. Since $k$ is even and $k\geq 6$, the graph $D$ contains at most two components and for the case that it contains exactly two components, the vertices $v_5$ and $v_2$ belong to different components of $D$. Therefore all components of $D$ have vertices with lists larger than their degrees in $D$, which implies that $D$ is $L$-choosable by Corollary \ref{cor-list}.
\end{proof}

%%%

%%%

%%%

%%%
%%%

%%%

%%%%%%%%%%%%%%%%%%%%%%%%%%%

%%%

%%%

%%%

%%%

%%%

%%%

%%%

%%%

%%%

%%%

%%%

%%%

%%%

%%%%%%%%%%%%%%%%%%%%%%

%%%

%%%

%%%

%%%

%%%

%%%

%%%

%%%

%%%%%%%%%%%%%%%%%%%%%%%%%%%%%%%
\section{Proof of Theorem~\ref{main}}\label{proof}

\begin{proof}
Let $G$ be an edge-minimal planar graph with $wd_3(G)>6$. By Lemma \ref{4-4}, the $4^+$-vertices of $G$ form an independent set in $G$.  Let $A_4$ be the set of vertices of degree at least 4 in $G$. Let $A^*_3$ be the set of vertices $v$ of degree 3 in $G$ having neighbors $u_1,u_2,u_3$ that satisfy the following properties:
 
 \begin{itemize}
 \item $d(u_1)=d(u_2)=3$;
 \item each of $u_1$ and $u_2$ has two $4^+$-neighbors;
 \item all neighbors of $u_3$ have degree 3.
 \end{itemize}
 
%Let $B_3$ be a maximal subset of $A_3^*$ in such a way that no pair of vertices of $B_3$ are adjacent in $G$ and no pair of them have a common $4^+$-neighbor in $G$. 
 
%  with $|N^2(v)|=6$ such that exactly four vertices in $N^2(v)$ are $4^+$-vertices and the other two vertices, say $u$ and $u'$, have degree $3$ in $G$,  where $u$ and $u'$ have the additional property that they are adjacent to a neighbor $v'$ of $v$.   

 For each vertex $w$ of $G$, choose $N^*(w)$ to be $\rm{min}\{d(w),3\}$ vertices on $N(w)$  in such a way that $|N(w)\cap A_3^*|$ is as small as possible. In case we have several options to choose $N^*(w)$ under this condition, we choose a set whose induced subgraph in $G$ has the maximum number of edges. 
 
  Let $G'$ be an auxiliary graph of $G$ having the same vertex set as $G$. For each vertex $v$ in $G$, make the vertices in $N^*(v)$ pairwise adjacent in $G'$. Note that by the structure of $G'$, any proper coloring of $G'$ corresponds to a 3-weak-dynamic coloring of $G$. Thus it is enough to prove that $\chi(G')\leq 6$.

Successively remove vertices $v$ in $V(G)-(A_4\cup A_3^*)$ from $G$ and instead make all vertices in $N_G(v)\cap (A_4\cup A_3^*)$ pairwise adjacent.  Let $H$ be the resulting graph. Each of these operations  preserves planarity, because it corresponds to adding cords to two or three faces of a planar graph and then removing a vertex. Also note that none of the edges added via this type of operation intersect, because their corresponding cords in $G$ are non-intersecting.  Therefore $H$ is planar.

If $u$ and $v$ are $4^+$-vertices in $G$ having a common neighbor $w$, then by the structure of $A_3^*$ and by Lemma \ref{4-4} we have $w\in V(G)-(A_3^*\cup A_4)$. Similarly, if $u\in A_4$ and $v\in A_3^*$ have a common neighbor $w$ in $G$, then $w\in V(G)-(A_3^*\cup A_4)$. Hence $H$ contains all the edges of $G'$ having at least one endpoint in $A_4$.

Since $H$ is planar, by the Four Color Theorem there exists a proper coloring $c:V(H)\rightarrow \{1,2,3,4\}$.  For any vertex $v\in A_4$, define $c^*(v)=c(v)$. Since $G'[A_4]\subseteq H$, the coloring $c^*$ is a proper coloring of $G'[A_4]$. To finish the proof we aim to extend $c^*$  to a proper coloring of $G'$ using colors in $\{1,\ldots,6\}$.

  For each $v$ in $V(G')$, let $N_4(v)=N_{G'}(v)\cap A_4$.  For each vertex $v$ in $V(G')-A_4$, we define $L(v)=\{1,\ldots,6\}-c^*(N_4(v))$.  Note that all vertices in $V(G')-A_4$ have degree at most 3 in $G$, and that by the choice of $N^*$, each $3$-vertex of $G$ has degree at most 6 in $G'$.  We already have a proper coloring of $G'[A_4]$ using four colors $\{1,2,3,4\}$. We aim to extend this coloring to a proper coloring of $G'$. Hence let $G''=G'-A_4$.  Note that if $G''$ is $L$-choosable, then we obtain an extention of the proper coloring of $G'[A_4]$  to a proper coloring of $G'$ using colors $\{1,\ldots,6\}$. Therefore for the remaining of the proof our aim is to prove that $G''$ is $L$-choosable.  
  
 Since $d_{G'}(v)\leq 6$ for each vertex $v$ in $V(G')-A_4$, we have $|L(v)|\geq d_{G''}(v)$. % Hence we have  list assignments on all vertices of $G''$ having the property that each vertex has a list with  size  at least the size of its degree in $G''$. 
 If any component of $G''$ has a vertex whose list size is greater than its degree, or if it has a block that is not a clique or odd cycle, then by Theorem \ref{list} and Corollary \ref{cor-list}  $G''$ is $L$-choosable, as desired.  Therefore let $C^*$ be a component of $G''$ whose vertices have list size equal to their degrees in $G''$ and whose blocks are complete graphs or odd cycles.

If $d_{G'}(v)\leq 5$, then $|L(v)|>d_{G''}(v)$. Hence $C^*$ does not contain such a vertex $v$. This simple observation implies that:

\begin{itemize}
\item $C^*$ contains no vertex $u$ whose degree is 2 in $G$;
\item $C^*$ contains no vertex $u$ such that $u$ has a 2-neighbor in $G$;
\item $C^*$ contains no vertex $u$ that is inside a 4-cycle  in $G$;
\item $C^*$ does not contain a vertex $u$ such that $u$ is a 3-vertex of $G$, it has a $4^+$-neighbor $u'$ in $G$, and $u\not\in N^*(u')$.
\end{itemize}

\noindent Also note that 

\begin{itemize}
\item $C^*$ contains no vertex $u$ of $A_3^*$,
\end{itemize}

\noindent because otherwise using the fact that $c$ is a proper coloring of $H$ using only 4 colors, we know that the four vertices in $N_{G'}(u)\cap A_4$  have at most three distinct colors under $c$. As a result, $|L(v)|\geq 3$ while $d_{G''}(v)\leq 2$.
 
 %We also have that
 
 %\begin{itemize}
 %\item  $C^*$ does not contain a vertex $v$ such that $v$ has at least 5 neighbors in $G'$ each of which belongs to $A_4$.
 %\end{itemize}

%Because if such a vertex $v$ exists, then $v$ has at most 2 neighbors in $G''$. As a result $|c(N_{4}(v))|\leq 4$ implies $|L(v)|>2\geq d(v)$, which is not accepted.  

Let $B$ be a pendant block of $C^*$.  By the choice of $C^*$ the block $B$ is a complete graph or an odd cycle. Note that since each vertex of $A_4$ has a color in $\{1,2,3,4\}$, each vertex of $G''$ gets a list of size at least $2$. Therefore no vertex in $B$ has degree $1$.  Hence $B$ contains at least three vertices.

We consider three cases.

\begin{description}
\item [Case 1:]  $B$ is an odd cycle.

Let the cycle $B$ be $u_1,u_2,\ldots,u_r$.  Therefore for each pair of vertices $u_i$ and $u_{i+1}$, there exists a vertex $v_i$ in $G$ such that $u_i$ and $u_{i+1}$ are neighbors of $v_i$ in $G$. Therefore $u_1v_1,v_1u_2,u_2v_2,v_2v_3,\ldots,u_{r}v_r,v_ru_1$ are all edges in $G$.

Let $r\geq 5$.
For each $i$, if $v_i$ has degree at least 4 in $G$, then by the construction of $G'$ and since all neighbors of $4^+$-vertices in $G$ are $3^-$-vertices, $u_i$ would be inside a triangle in $B$. Hence all vertices $v_1,\ldots,v_r$   have degree 3 in $G$.  If $r\geq 4$ and $v_i=v_{i+1}$ for some $i$, then $N^*(v_i)=\{u_i,u_{i+1},u_{i+2}\}$. As a result, the vertex $u_i$ has neighbors $u_{i-1},u_{i+1},u_{i+2}$ in $B$. This is a contradiction since $B$ is a cycle.
Otherwise, recall  that $u_1,\ldots,u_r$ are distinct vertices. Note that $u_1v_1u_2v_2\ldots u_{r}v_ru_1$ is a closed walk in $G$. Since $u_i$s are distinct and since $v_i\neq v_{i+1}$ for all $i$, no edge is repeated immediately in the closed walk. As a result of Proposition~\ref{pro-walk}, there exists a cycle in $G$ containing a subset of $\{u_1,\ldots,u_r\}\cup \{v_1,\ldots,v_r\}$. Hence we find a cycle $C$ in $G$ all whose vertices have degree 3.   This is a contradiction with Lemma \ref{lemma:3-RegularkCycle-remaining-cases}.

Now suppose $r=3$. If $v_1,v_2,$ and $v_3$ are distinct vertices, then similar to the above argument we obtain a contradiction by finding a cycle in $G$ all whose vertices have degree 3. Hence suppose $v_1=v_2$. Therefore $v_1$ is adjacent to $u_1,u_2,$ and $u_3$ in $G$. Recall that $B$ is a pendant block of $C^*$. Therefore at least two vertices of $B$ have degree 2 in $C^*$. As a result, at least two vertices in $\{u_1,u_2,u_3\}$ have four $4^+$-vertices on their second neighborhood. In fact, those two vertices  belong to $A_3^*$, because each of them has a neighbor ($v_1$) all of whose neighbors are $3^-$neighbors and has two other neighbors whose neighbors are $4^+$-vertices. This is   a contradiction because as we argued above $C^*$ contains no vertex of $A_3^*$.

%Hence $B$ contains a vertex $v$ in $B$ such that $v$ is not a cut-vertex of $C^*$. We know that $v$ is a $3$-vertex in $G$. Let $N_G(v)=\{u_1,u_2,u_3\}$. We have $v\in N'(u_1)\cap N'(u_2)\cap N'(u_3)$, since otherwise the list we defined on $v$ would have a size larger than 2, which contradicts the choice of $C^*$. Moreover there are exactly 6 vertices in $N'(u_1)\cup N'(u_2)\cup N'(u_3)-\{v\}$, since otherwise $v$ is in a 4-cycle, which is not accepted by the above argument.  Since $v$ has degree $2$ in $G''$, exactly four vertices in $N'(u_1)\cup N'(u_2)\cup N'(u_3)-\{v\}$ belong to $A_4$. Let $w$ and $w'$ be the two $3$-vertices in $N'(u_1)\cup N'(u_2)\cup N'(u_3)-\{v\}$. Note that $w$ and $w'$ both belong to $B$.

%First suppose that $w\in N'(u_1)$ and $w'\in N'(u_2)$. Let $z$ and $z'$ be the $4^+$-neighbors of $u_1$ and $u_2$, respectively. In this case we obtain the Configuration 4-3-3-3-4, which is a contradiction.  Hence we may suppose that $w$ and $w'$ both are neighbors of $u_1$ in $G$. As a result, $v\in A_3^*$, a contradiction. 

\item [Case 2:]  At least one vertex in $V(B)$ is part of a 3-cycle in $G$. 

Let $wv_1v_2$ be a triangle in $G$ such that $\{w,v_1,v_2\}\cap V(B)\not=\emptyset$. By Lemma \ref{3-uniform triangle}, we may suppose that $d_G(w)\geq 4$ and $d_G(v_1)=d_G(v_2)=3$. Recall that vertices of $B$ are $3$-vertices in $G$.  Hence either $v_1$ and $v_2$ both belong to $V(B)$ or only one of them belongs to $V(B)$. Let $N_G(v_1)-\{w,v_2\}=\{v'_1\}$ and $N_G(v_2)-\{w,v_1\}=\{v'_2\}$.  We consider two subcases.

\textbf{Subcase 1.} $v_1\in V(B)$ and $v_2\in V(B)$. By Lemmas \ref{lemma:C_3withDegree3Neighbors} and \ref{lemma:C_3withDegree5top} we may suppose that $d_G(v'_2)\geq 4$. By the construction of $G''$, there exists a neighbor $v_3$ of $w$ such that $N^*(w)=\{v_1,v_2,v_3\}$.   Lemmas~\ref{lemma:C_3bottomadjacentC_3} and~\ref{lemma:edgeadjacentC_3} imply that $v'_1$, $v'_2$, and $v_3$ are distinct vertices.

%As we argued above, we have  $d_G(v'_1)\geq 3$ and $d_G(v'_2)\geq 3$.  First suppose that $d_G(v'_1)=3$ and $d_G(v'_2)=3$. Note that at least one of $v_1$ and $v_2$ is not a cut-vertex in $C^*$. If $v_1$ is not a cut-vertex in $C^*$, then $v'_2$ is a vertex in $B$. As a result $v_1$ has degree at least 3 in $B$, which means $B$ has to be a complete graph. Hence $v_2$ and $v'_2$ must have a common neighbor in $G$. Therefore $v'_2w\in E(G)$ or $v'_2v_1\in E(G)$, both of which give us a contradiction. Similarly, if $v_2$ is a cut-vertex we get a contradiction.

%Hence we may suppose that $d_G(v'_2)\geq 4$.

 Since $d_{G}(v'_2)\geq 4$ by the construction of $G'$, the vertex $v'_2$ has two neighbors $v_4$ and $v_5$ in $G$ such that $N^*(v'_2)=\{v_2,v_4,v_5\}$. Note that since $G$ has no 4-cycle containing a vertex in $C^*$, the vertices $v_4$ and $v_5$ are distinct from $v_1$ and $v_3$.

The vertex $v_2$ is adjacent to $v_4$ and $v_5$ in $C^*$. If $v_2$ is not a cut-vertex of $B$ or if $v_4$ and $v_5$ belong to $B$, then  $B$ contains at least 5 vertices ($\{v_1,\ldots,v_5\}$). Hence $B$ cannot be a cycle, because $v_2$ is adjacent to $v_1,v_3,v_4,v_5$ in $B$. Therefore $B$ is a complete graph. Hence vertices $v_4$ and $v_5$ must be adjacent to $v_1$ in $B$. Equivalently, $v_4$ and $v_5$ must have common neighbors with $v_1$ in $G$. If $v_4w\in E(G)$ or $v_5w\in E(G)$, then $v_2$ belongs to a 4-cycle in $G$, which is not accepted. Hence we must have $v_4v'_1\in E(G)$ and $v_5v'_1\in E(G)$. This is a contradiction, because $v'_2v_4v'_1v_5v'_2$ forms a 4-cycle in $G$. 

Hence  $v_2$ must be a cut-vertex in $C^*$. If $v_4$ is a vertex of $B$, knowing that $v_4$ is not a cut-vertex of $B$, then we conclude that $v_5$ belongs to $B$. But we argued above that the case $v_4\in V(B)$ and $v_5\in V(B)$ cannot happen. Hence  none of the vertices $v_4$ and $v_5$  belongs to $B$.

We use a similar argument as above to show that $d_G(v'_1)=3$.  If $d_G(v'_1)\geq 4$,  then let $N^*(v'_1)=\{v_1,v_6,v_7\}$. Since $v_1$ is not a cut-vertex of $C^*$, the vertices $v_6$ and $v_7$ belong to $B$. Hence $B$ contains at least five vertices ($\{v_1,v_2,v_3,v_6,v_7\}$). Hence $B$ cannot be a cycle, because $v_1$ is adjacent to $v_2,v_3,v_6,v_7$ in $B$. Therefore $B$ is a complete graph. Hence vertices $v_6$ and $v_7$ must be adjacent to $v_2$ in $B$. Equivalently, $v_6$ and $v_7$ must have common neighbors with $v_2$ in $G$. If $v_6w\in E(G)$ or $v_7w\in E(G)$, then $v_1$ belongs to a 4-cycle in $G$, which is not accepted. Hence we must have $v_6v'_2\in E(G)$ and $v_7v'_2\in E(G)$. This is a contradiction, because $v'_1v_6v'_2v_7v'_1$ forms a 4-cycle in $G$.  Hence we have $d_G(v'_1)\leq 3$, and so by Lemma~\ref{lemma:AdjacentDegree2s}, we have $d_G(v'_1)=3$.

 Since $C^*$ has no vertex in $A_3^*$, the vertex $v'_1$ does not have two $4^+$-neighbors in $G$, otherwise $v_1\in A_3^*$.  Hence $v'_1$ must have at least one other 3-neighbor $v_6$ beside $v_1$. The vertex $v_6$ is adjacent to $v_1$ in $B$, and as a result it must also be adjacent to $v_2$ in $B$. Therefore $v_6$ must have a common neighbor with $v_2$ in $G$ that belongs to $N^*(v_2)$. That common neighbor is not $w$, because otherwise we find a 4-cycle containing $v_1$ in $G$. Hence $v_6$ must    belong to $N^*(v'_2)$. In other words $v_6=v_4$ or $v_6=v_5$. But this is a contradiction, because $v_6$ is a vertex of $B$ while $v_4$ and $v_5$ are not vertices of $B$.

 %With no loss of generality suppose $v_6=v_4$. Therefore $v_1,v_2,v_3,v_4$ belong to $B$. Since $v_4$ is not a cut-vertex of $B$ and since $v_5$ is a neighbor of $v_4$ in $B$, we have $v_5$ in $B$ as well, contradicting the above argument.

\textbf{Subcase 2.} $v_1\in V(B)$ but $v_2\not\in V(B)$. By the construction of $G'$, there exist  neighbors $v_3$ and $v_4$ of $w$ such that $N^*(w)=\{v_1,v_3,v_4\}$.    If $v_3v_4\in E(G)$, then we can repeat Subcase 1 for the triangle $wv_3v_4$. Hence suppose $v_3v_4\not\in E(G)$. Therefore by the choice of $N^*(w)$, we have $v_2\in A_3^*$, $v_3\not\in A_3^*$, and $v_4\not\in A_3^*$, since otherwise $\{v_1,v_2,v_3\}$ or $\{v_1,v_2,v_4\}$ would give us a better option for $N^*(w)$, according to the choice of $N^*(w)$.

Since $v_2\in A_3^*$, the vertex $v'_2$ has degree 3 in $G$ and has two $4^+$-neighbors in $G$. By the same reason $d_G(v'_1)\geq 4$.  Let $N^*(v'_1)=\{v_1,v_5,v_6\}$. Note that we know $v_1\in N^*(v'_1)$, since otherwise the vertex $v_1$ has a list of size larger than its degree in $G''$. We have $\{v_5,v_6\}\cap \{v_2,v_3,v_4\}=\emptyset$, since otherwise $G$ contains a 4-cycle containing $v_1$, which is not accepted. Therefore according to the adjacencies we have determined so far in $G$, the vertex $v_1$ has neighbors $\{v'_2,v_3,\ldots,v_6\}$ in $C^*$. Therefore $d_{C^*}(v_1)=5$.

Let $v_7$ and $v_8$ be the $4^+$-neighbors of $v'_2$.  Since vertex $v'_2$ has two $4^+$-neighbors and since $v'_2$ belongs to $C^*$ (because it is adjacent to $v_1$ in $C^*$), we must have $v'_2\in N^*(v_7)$ and $v'_2\in N^*(v_8)$, since otherwise the list of $v'_2$ in $G''$ has size larger than its degree in $G''$, which is not accepted. Therefore $d_{C^*}(v'_2)=5$.

Let $N_G(v_3)=\{w,v'_3,v''_3\}$ and $N_G(v_4)=\{w,v'_4,v''_4\}$. If the neighbors of $v_3$ in $C^*$ are only $v_1$ and $v_4$, then $v_3$ has to be a vertex in $A_3^*$, which is not accepted. If $v_3$ has at most one more neighbor besides $v_1$ and $v_4$ in $C^*$, then we must have $d_G(v'_3)=d_G(v''_3)=3$, one vertex in $\{v'_3,v''_3\}$ has exactly one $3$-neighbor $x$, and one vertex in $\{v'_3,v''_3\}$ has two $4^+$-neighbors. When $x\neq w$ we get a contradiction with Lemma \ref{Fig-4-3-3-4} and when $x=w$ we get a contradiction with Lemmas \ref{lemma:C_3withDegree3Neighbors} and \ref{lemma:C_3withDegree5top}. Therefore $d_{C^*}(v_3)\geq 4$.  By a similar argument, we have $d_{C^*}(v_4)\geq 4$, $d_{C^*}(v_5)\geq 4$, and $d_{C^*}(v_6)\geq 4$.

%If $v'_3$ has exactly one neighbor $x$ of degree at least 4 and $x\neq w$, then the vertices $x,v'_3,v_3,w$ form a configuration as the one in Lemma \ref{lemma:4-3-3-4Configuration} which is a contradiction. If $v'_3$ has exactly one neighbor $x$ of degree at least 4, $x=w$, and $d_G(v''_3)=3$ then we obtain a contradiction with Lemmas \ref{lemma:C_3withDegree3Neighbors} and \ref{lemma:C_3withDegree5top}. Also note that none of the vertices $v_3$ and $v_4$ belongs to $A_3^*$, because they belong to $C^*$.   Therefore  $d_G(v'_3)\geq 4$ or $d_G(v''_3)\geq 4$ or at least one of them has two neighbors of degree 3, since otherwise a situation we just described happens. In all the three cases we have at least 4 neighbors for $v_3$ in $C^*$, which implies $d_{C^*}(v_3)\geq 4$. Similarly we have $d_{C^*}(v_4)\geq 4$.  By a similar argument we have $d_{C^*}(v_5)\geq 4$ and $d_{C^*}(v_6)\geq 4$.

By the above arguments, the vertices $v_1,v'_2,v_3,v_4,v_5,v_5$ belong to $C^*$ and all of them have degree at least 4 in $C^*$. We know moreover that $N_{C^*}(v_1)=\{v'_2,v_3,v_4,v_5,v_6\}$ and the vertex $v_1$ is a vertex of the block $B$. Hence $B$ has 5 or 6 vertices. Since $v_1,v_3,v_4$ and $v_1,v_5,v_6$ form triangles in $C^*$, we conclude that either $V(B)=\{v_1,v'_2,v_3,v_4,v_5,v_6\}$ or  $V(B)=\{v_1,v_3,v_4,v_5,v_6\}$. In the both cases $B$ cannot be an odd cycle, so it  is a complete graph. 

Hence $v_3$ and $v_5$ have a common neighbor $z$ in $G$. Also $v_3$ and $v_6$ have a common neighbor $z'$ in $G$. We have $z\neq z'$ and $\{z,z'\}\cap \{w,v_1,\ldots,v_6,v'_1,v'_2\}$, since otherwise a 4-cycle containing a vertex of $B$ exists in $G$ or Subcase 1 can be applied. Similarly there are disjoint vertices $y$ and $y'$ in $G$ such that $y$ is a common neighbor of $v_4$ and $v_5$ in $G$, $y'$ is a common neighbor of $v_4$ and $v_6$ in $G$, and  $\{y,y'\}\cap \{w,v_1,\ldots,v_6,v'_1,v'_2\}$. We also have $\{z,z'\}\cap \{y,y'\}=\emptyset$, since otherwise $v_3$ or $v_5$ is inside a 4-cycle in $G$. 

Now the vertices $w,v_5,v_6$ and $v_1,v_3,v_4$ are the branch vertices of a $K_{3,3}$-minor in $G$, which implies $G$ is not planar, a contradiction.

\item [Case 3:]  $B$ is a complete graph. 

By Case 1 we may suppose $B$ is a complete graph with four, five,  six, or seven vertices, as each vertex in $G''$ has degree at most 6. Since $B$ is a pendant block, in $G''$ all but at most one vertex of $B$ has all its neighbors in $V(B)$.  Let $v$ be one of the vertices of $B$ all whose neighbors in $G''$ are in $V(B)$, i.e. $v$ is not a cut-vertex of $C^*$. Let $u_1,u_2,u_3$ be the neighbors of $v$ in $G$. By Case 2, $\{u_1,u_2,u_3\}$ forms an independent set in $G$. 

We consider three subcases. 

\textbf{Subcase 1.} Two of the neighbors of $v$ in $B$, say $w_1$ and $w_2$, are neighbors of $u_1$ in $G$, and two of the neighbors of $v$ in $B$, say $w_3$ and $w_4$, are neighbors of $u_2$ in $G$.

By Case 2, we may suppose that $\{w_1,w_2,w_3,w_4\}\cap \{u_1,u_2,u_3\}=\emptyset$. Since $G$ is planar, we may suppose that the vertices $w_1,\ldots,w_4$ appear in the counterclockwise direction in the drawing of $G$.  Note that  $w_1,\ldots,w_4$ have degree 3 in $G$. Since $B$ is a complete graph, the four vertices $w_1,\ldots,w_4$ are pairwise adjacent in $B$, and hence each pair of them must have a  common neighbor in $G$.

 Let $y_1$ be the common neighbor of $w_1$ and $w_3$ in $G$. We have $y_1\neq w_4$, since otherwise $w_3w_4\in E(G)$ and Case 2 can be applied on the triangle $u_2w_3w_4$. Similarly $y_1\neq w_2$. Hence all the vertices $v,u_1,u_2,w_1,w_2,w_3,w_4,y_1$ are distinct. Now consider the cycle $C':\ vu_1w_1y_1w_3u_2v$. Since the vertices $w_1,\ldots,w_4$ are in counterclockwise direction,  the cycle $C'$ separates the  vertex $w_2$ from the vertex $w_4$  in $G$.  In order to have a common neighbor for $w_2$ and $w_4$ in $G$,  both of $w_2$ and $w_4$  have to be adjacent to a vertex $x$ in the cycle $C'$. We have $x\neq v$, because the only neighbors of $v$ in $G$ are $u_1,u_2,u_3$. We have $x\neq u_1$, $x\neq u_2$, and $x\neq y_1$,  since otherwise $G$ contains a 4-cycle containing $w_2$ or $w_4$, which is not accepted. We have $x\neq w_1$ and $x\neq w_3$, because otherwise Case 2 can be applied. Therefore this subcase does not happen.

\textbf{Subcase 2.} Two of the neighbors of $v$ in $B$, say $w_1$ and $w_2$, are neighbors of $u_1$ in $G$, and one of the neighbors of $v$ in $B$, say $w_3$, is a neighbor of $u_2$ in $G$.

Since $G$ is planar, we may suppose that the vertices $w_1,w_2,w_3$ appear in the counterclockwise direction in  $G$. Note that when $d_B(v)=6$ or $d_B(v)=5$, Subcase 1 can be applied to get a contradiction. Hence we may suppose that $d_B(v)\leq 4$. By Subcase 1, we may also suppose that $u_2$ has a neighbor of degree at least $4$. As a result, $d_G(u_2)=3$.  By a similar argument we have $d_G(u_3)=3$. Let $z$ be the $4^+$-neighbor of $u_2$. 

If $d_G(u_1)\geq 4$, then $u_1,v,u_2,z,u_3,w_3$ form a configuration as of Lemma \ref{lemma:4-3-3-4Configuration}, which is a contradiction. Therefore we have $d_G(u_1)=3$. The vertices $w_1$ and $w_3$ must have a common neighbor $y_1$ in $G$. By Case 2, the vertex $y_1$ is different from vertices $w_2$ and $z$. Therefore the vertices $v,u_1,u_2,w_1,w_2,w_3,z,y_1$ are all distinct vertices in $G$. If $d_G(y_1)\leq 3$, then $y_1w_1u_1vu_2w_3y_1$ forms a cycle of all $3^-$-vertices, which contradicts Lemma \ref{lemma:3-RegularkCycle-remaining-cases}. Hence $d_G(y_1)\geq 4$. 

By the construction of $G''$, the vertex $y_1$ has a neighbor $w_4$ in $G$ such that $w_4$ is adjacent to $w_1$ and $w_3$ in $B$, i.e. $N^*(y_1)=\{w_1,w_3,w_4\}$. Note that $w_4\neq w_2$, since otherwise a 4-cycle containing $w_2$ exists in $G$. On the other hand since $B$ is a complete graph, $w_4$ must be in the second neighborhood of $v$. Therefore $w_4$ must be adjacent to $u_3$.  

If $w_3$ has only one $4^+$-neighbor in $G$ (the vertex $y_1$), then  $y_1,w_3,u_2,z$ form a configuration as the one in Lemma \ref{Fig-4-3-3-4}, which is a contradiction. Similarly, if $w_4$ has only one $4^+$-neighbor in $G$ (the vertex $y_1$), then the vertices  $y_1,u_4,u_3$, and the $4^+$-neighbor of $u_3$  form a configuration as the one in Lemma \ref{Fig-4-3-3-4}, which is a not accepted. Therefore both of $w_3$ and $w_4$ have two $4^+$-neighbors in $G$. As a result, each of them has degree $5$ in $C^*$. We can repeat Subcase 1 for a vertex in $\{w_3,w_4\}$ that is not a cut-vertex of $C^*$.

%Now consider the cycle $C':\ vu_1w_1y_1w_3u_2v$. This cycle separates the  vertex $w_2$ from the vertex $w_4$.  In order to have a common neighbor for $w_2$ and $w_4$ in $G$,  both of $w_2$ and $w_4$  have to be adjacent to a vertex $x$ in the cycle $C'$. We have $x\neq v$, because the only neighbors of $v$ in $G$ are $u_1,u_2,u_3$. We have $x\neq u_1$, $x\neq u_2$, and $x\neq y_1$,  since otherwise $G$ contains a 4-cycle containing $w_2$ or $w_4$, which is not accepted. We have $x\neq w_1$ and $x\neq w_3$, since otherwise Case 2 can be applied. Therefore this subcase does not happen.

\textbf{Subcase 3.} Exactly one neighbor of $v$ in $B$, say $w_1$ is a neighbor of $u_1$ in $G$,  exactly one neighbor of $v$ in $B$, say $w_2$ is a neighbor of $u_2$ in $G$, and exactly one neighbor of $v$ in $B$, say $w_3$ is a neighbor of $u_3$ in $G$.

Therefore, by Subcases 1 and 2, we may suppose that each of $u_1,u_2$, and $u_3$ has a $4^+$-neighbor in $G$. Suppose $z_1$ is the $4^+$-neighbor of $u_1$ in $G$, $z_2$ is the $4^+$-neighbor of $u_2$ in $G$, and $z_3$ is the $4^+$-neighbor of $u_3$ in $G$. Hence $d_G(u_1)=d_G(u_2)=d_G(u_3)=3$. Note that in this case $B$ is a complete graph with vertices $w_1,w_2,w_3,$ and $v$. Hence $w_1$ and $w_2$ must have a common neighbor, say $y_1$, in $G$.

If $d_G(y_1)=3$, then $vu_1w_1y_1w_2u_2v$ is a cycle in $G$ all whose vertices have degree 3, a contradiction with Lemma \ref{lemma:3-RegularkCycle-remaining-cases}.  Hence we must have $d_G(y_1)\geq 4$. Since $|N^*(y_1)|=3$,  all vertices in $N^*(y_1)$ have degree at most 3, and since $B$ has only four vertices, the vertex $y_1$ must be adjacent to $w_3$ in $G$. Recall that at most one vertex in $\{w_1,w_2,w_3\}$ is a cut-vertex of $C^*$. With no loss of generality suppose $w_1$ is not a cut-vertex of $C^*$. Now subcase 2 can be applied on $w_1$ to get a contradiction.
\end{description}
\end{proof}

\section{Future Work}

At the moment, we know of no planar graph with 3-weak-dynamic number 6. 
However, there are planar graphs with 3-weak-dynamic number 5, as we can see in Figure \ref{Fig-3-weak-dynamic-5}. Therefore  the best general upper bound for 3-weak-dynamic number of planar graphs  is either 5 or 6. 

\begin{question}
Are there planar graphs that have 3-weak-dynamic number 6?
\end{question}

\begin{figure} [H]
\begin{subfigure}{0.5\textwidth}
\begin{center}
\begin{tikzpicture}
\draw (0,0) -- (1.5,0) -- (1.5,2) -- (0,2) -- (0,0);
\draw (0,0) to[bend left] (0,2);
\draw (1.5,0) to[bend right] (1.5,2);
\filldraw (0,2) circle (3pt) -- (1.5,1) circle (3pt);
\filldraw (0,1) circle (3pt) -- (1.5,0) circle (3pt);
\filldraw (0,0) circle (3pt);
\filldraw (1.5,2) circle (3pt);
\end{tikzpicture}
\end{center}
\caption{$wd_3(G)=5$}
\end{subfigure}
\begin{subfigure}{0.5\textwidth}
\begin{center}
\begin{tikzpicture}
\foreach \x in {1,...,7}
{
    \filldraw (360/7*\x-90/7:1.25) circle (3pt);
    \draw (360/7*\x-90/7:1.25) -- (360/7*\x+270/7:1.25);
}
\foreach \x in {3,5}
    \draw (360/7*\x-90/7:1.25) -- (360/7*\x+90:1.25);
\draw (2070/7:1.25) .. controls (2310/7:1.875) and (30/7:1.875) .. (270/7:1.25);
\end{tikzpicture}
\end{center}
\caption{$wd_3(G)=5$}
\end{subfigure}
\caption{Graphs with 3-weak-dynamic number 5.}\label{Fig-3-weak-dynamic-5}
\end{figure}
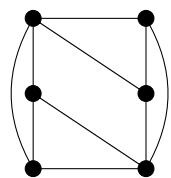
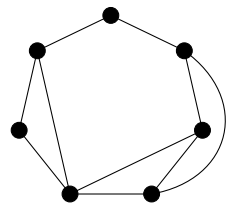

%%%%%%%%%%%%%%%%%%%%%%%%%%%%%%

\end{document}